\documentclass[reqno]{amsart}
\usepackage{amssymb}
\usepackage{mathrsfs}
\usepackage[backend=biber]{biblatex} 
\usepackage{hyperref}
\hypersetup{colorlinks,linkcolor={blue},citecolor={red},urlcolor={blue}}
\renewbibmacro{in:}{}

\newtheorem{theorem}{Theorem}[section]
\newtheorem{lemma}[theorem]{Lemma}
\newtheorem{cor}[theorem]{Corollary}
\newtheorem{prop}[theorem]{Proposition}
\theoremstyle{definition}
\newtheorem{definition}[theorem]{Definition}

\theoremstyle{remark}
\newtheorem{remark}[theorem]{Remark}

\newcommand{\qbinomial}[2]{\genfrac[]{0pt}{}{#1}{#2}_q}
\newcommand{\dontprint}[1]\relax

\newcommand{\gl}{{\mathfrak{gl}}}
\newcommand{\vac}{|\mathrm{vac}\rangle}

\newcommand{\IW}{\mathbb{W}}
\newcommand{\la}{\lambda}

\newcommand{\fg}{\mathfrak{g}}
\newcommand{\fh}{\mathfrak{h}}
\newcommand{\fsl}{\mathfrak{sl}}

\begin{document}
\author{Rea Dalipi}
\address{Department of Mathematics, University of Geneva,
Rue du Conseil-Général 7-9,
1205 Geneva, Switzerland.
}
\email{rea.dalipi@unige.ch}
\author{Giovanni Felder}
\address{Department of Mathematics, ETH Zurich, 8092 Zurich, Switzerland}
\email{giovanni.felder@math.ethz.ch}

\makeatletter
\let\@wraptoccontribs\wraptoccontribs
\makeatother

\contrib[with an appendix by]{Anfisa Gurenkova}
\address{Krichever Center for Advanced Studies, Skolkovo Institute of Science and Technology, Moscow, Russia}
\email{anfisa.gurenkova@skoltech.ru}

\title{Howe duality and dynamical Weyl group}

\begin{abstract}
  We give a fermionic formula for $R$-matrices of exterior powers of
  the vector representations of $U_q(\widehat{ \mathfrak{gl}}_N)$ and
  relate it to the dynamical Weyl group of Tarasov--Varchenko and
  Etingof--Varchenko, via a Howe
  ($\mathfrak{gl}_N,\mathfrak{gl}_M)$-duality.  In the limit $N\to\infty$
  we obtain $R$-matrices for Fock spaces.  As a consequence of our result we
  obtain a dynamical action of the Weyl group on integrable
  $U_q\mathfrak{gl}_M$-modules, extending the known action on zero
  weight spaces.  In an Appendix by Anfisa Gurenkova it is shown that
  the latter property also holds if we replace $\mathfrak{gl}_M$ by a
  general symmetrizable Kac--Moody algebra.
\end{abstract}
\maketitle
\section{Introduction}\label{s-1}
The symmetric groups play two different roles in the representation
theory of general linear groups.  On one side, let us call it the Weyl
side, the symmetric group $S_N$ of permutations of $N$ letters is the
subgroup of $\mathit{GL}_N$ of permutation matrices and thus acts on
any representation of $\mathit{GL}_N$.  On the other side, say the
Schur side, $S_M$ is the group of permutation of factors of the $M$th
tensor power $U^{\otimes M}=U\otimes\cdots\otimes U$ of a
representation $U$ of $\mathit{GL}_N$ and the action of $S_M$ commutes
with the natural action of $\mathit{GL}_N$ on $U^{\otimes M}$. We will
focus on the case where $U$ is an exterior power of the vector
representation, but let us be general for the time being.

Both roles admit a $q$-deformed version but the symmetric group needs
to be replaced by the braid group.  On the Weyl side one has an action
of the braid group $B_N$ on finite dimensional representations of the
quantum enveloping algebras $U_q\mathfrak{gl}_N$ deforming the above
action of the symmetric group. This is a special case of the action on
integrable representations of the Artin braid group corresponding to
the Weyl group of any semisimple Lie algebra \cite{Lusztig1993,
  LevendorskiiSoibelman1990}.  On the Schur side, the universal $R$-matrix
of $U_q\mathfrak gl_N$ defines a representation of the braid group
$B_M$ on the tensor power $U^{\otimes M}$ of a finite dimensional
representation $U$ of $U_q\gl_N$. The action of the generator $s_i$,
braiding the $i$th with the $i+1$st strand is given by the action of
the $R$-matrix on the $i$th and $i+1$st factor.

Both sides are degenerate cases of versions with parameters, called
spectral (or equivariant) parameters on the Schur side and dynamical
(or Kähler) parameters on the Weyl side. In terms of representation
theory, this generalization amounts to replacing $U_q\mathfrak{gl}_N$
by its affine version $U_q\widehat{\mathfrak{gl}}_N$. A feature of
the versions with parameters, whose clarification is a goal of this
paper, is that the symmetric group plays the leading role, rather than
the braid group that is relevant in the degenerate limit. The
$R$-matrices of quantum groups first appeared in statistical mechanics
and quantum field theory as solutions
$R(z)\in\operatorname{End}(U\otimes U)$ of the Yang--Baxter equation,
depending meromorphically on a complex spectral parameter $z$. The
Yang--Baxter equation is equivalent to the braiding relation
\eqref{e-3} for the braiding matrix $\check R(z)=P\circ R(z)$
obtained by composing the $R$-matrix with the permutation
$P(u\otimes v)=v\otimes u$ of factors. Moreover one has the inversion
(or unitarity) relation $\check R(z^{-1})\check R(z)=\mathrm{id}$.
One obtains, on the Schur side, a representation of the symmetric
group $S_M$ on $U^{\otimes M}$-valued meromorphic functions
$f(z_1,\dots, z_M)$ of $M$ variables, such that the generator
$s_i=(i,i+1)$ is mapped to
\[
s_if(z_1,\dots,z_N) = \mathrm{id}^{\otimes (i-1)}\otimes \check
 R(z_i/z_{i+1})\otimes \mathrm{id}^{\otimes (N-i-1)}
 f(z_1,\dots,z_{i+1},z_i,\dots,z_N).
\]
Representations of the braid group on $U^{\otimes M}$ appear in
the limits $z_i/z_{i+1}\to 0$, $i=1,\dots,M-1$ or $\infty$. Indeed $R(z)$
converges in the limit $z\to\infty$ to the solution of the Yang--Baxter equation
without spectral parameter corresponding to $U_q\gl_N$. The relation
$s_i^2=1$ is lost in this limit and one is left with a representation of the
braid group.

On the Weyl side, the discovery of a version of the quantum Weyl group
action with parameters came later
\cite{TarasovVarchenko2000,EtingofVarchenko2002} and is called
dynamical Weyl group.  More recently \cite{OkounkovSmirnov2022} a geometric
realization of a dynamical Weyl group action on the equivariant
$K$-theory of Nakajima varieties was found.  Let
$T\subset SL_M(\mathbb C)$ be the Cartan torus of diagonal matrices,
$P_M=\operatorname{Hom}(T,\mathbb C^\times)$ the weight lattice. Then
finite dimensional representations $U$ of $U_q\mathfrak{sl}_M$ (of
type I with $q$ not a root of unity) have a weight decomposition
$U=\oplus_{\mu\in P_M}U[\mu]$. The Weyl group $S_M$ acts on $T$ and
$P_M$, defining a action of the braid group via the canonical
homomorphism $B_M\to S_M$. A dynamical action of the braid group $B_M$
on an $\mathfrak{sl}_M$-module $U$ is a representation $\rho$ of the
braid group on rational $U$-valued function on $T$ of the form
\[
  (\rho(g)f)(t)=A_{g}(g^{-1}t)f(g^{-1}t),
\]
for some rational functions $t\mapsto A_g(t)\in\operatorname{End}(U)$,
compatible with the $T$-action on $U$ in the sense that
$A_g(t)(U[\mu])\subset U[g\mu]$.  The theory of intertwining operators
developed in \cite{TarasovVarchenko2000, EtingofVarchenko2002}, which works
for arbitrary simple Lie algebras, gives rise to a dynamical action of
the braid group with many interesting properties, such as universal
formulas for $A_g(t)|_{U[\mu]}$ as the action of an element of
$U_q\mathfrak{sl}_M$ independent of $U$ and a compatibility with
inclusions of Lie subalgebras, implying that $A_{s_i}(t)$ for
generators $s_i$ is obtained as the image of the corresponding element
of $U_q(\mathfrak{sl}_2)$ by the inclusion associated with the $i$th
simple root.

In this paper we take $U=\bigwedge V$, the direct sum of exterior
powers of the vector representation $V=\mathbb C^N$ of $U_q\gl_N$, and
relate the two actions of the symmetric group via quantum (skew) Howe
duality. The classical Howe duality \cite{Howe1992} is based on the
observation that on $\bigwedge(\mathbb C^N\otimes\mathbb C^M)$ we have
natural commuting actions of $\mathit{GL}_N$ and $\mathit{GL}_M$ whose
images span each other's commutant in the endomorphism ring. This
implies that endomorphisms of
$(\bigwedge V)^{\otimes M}=\bigwedge V\otimes\cdots\otimes\bigwedge V$
commuting with the $\mathit{GL_N}$ action on $V=\mathbb C^N$ are
spanned by a commuting action of $\mathit{GL}_M$. There is a also a
symmetric version of Howe duality, where exterior powers are replaced
by symmetric powers.  Quantum version of Howe dualities are known.
{}The case of symmetric powers was first observed and proved by Toledano Laredo in
\cite{ToledanoLaredo2002}. The skew-symmetric case was then considered in
\cite{CautisKamnitzerLicata2010} and \cite{LehrerZhangZhang2011}.  See
\cite{Stroppel2022} for recent developments of this story.  In Section
\ref{s-3} we present an approach based on the action of the Clifford
algebra which is more elementary than the categorification approach of
\cite{CautisKamnitzerMorrison2014}, see Proposition \ref{p-5}.

In particular one obtains a $U_q\gl_M$-action on
$(\bigwedge V)^{\otimes M}$ commuting with the diagonal
$U_q\gl_N$-action.  The upshot is that the action of the symmetric
group on $(\bigwedge V)^{\otimes M}$-valued functions by braiding
matrices with spectral parameter is realized by a dynamical Weyl group
action for the $U_q\gl_M$-action, after a suitable identification of
spectral with dynamical parameters.  Thus we connect the Schur with
the Weyl side. In the limit of infinite spectral/dynamical parameter
we recover the result of \cite{CautisKamnitzerLicata2010} relating
braid group action with quantum Weyl group action.

For $M=2$ we obtain an explicit formula for the $R$-matrix for pairs
of exterior powers of the vector representation in terms of
$U_q\gl_2$, see Theorem \ref{t-1}. These $R$-matrices were calculated
in \cite{DateOkado1994} by the fusion procedure and are expressed in
terms of projections onto irreducible subrepresentations of the tensor
product. Our alternative formula does not rely on the decomposition
into irreducible subrepresentations. It is a sum of terms, each with a
simple pole in the spectral parameters, and gives an interpretation of
the residues at the poles.  Also, it is given by a formula without
explicit dependence on $N$. In fact it has a well-defined limit
as $N\to\infty$ resulting in solutions of the Yang--Baxter equations
for fermionic Fock spaces, a.k.a.\ spaces of  semi-infinite exterior powers,
see Theorem \ref{t-Fock}.

When written in terms of Clifford algebra
generators, our formula for the $R$-matrix of
$\bigwedge^kV\otimes\bigwedge^kV$ converges in the limit
$q\to1,z\to 1$ to the formula proposed by Smirnov \cite{Smirnov2016}
in the case of the Yangian, see Theorem \ref{t-3}.  Finding a
conceptual proof of Smirnov's formula was an initial motivation of
this work. Note also that similar
constructions in this limit are known in the context of
Knizhnik--Zamolodchikov and dynamical equations
\cite{ToledanoLaredo2002,TarasovVarchenko2002,MukhinTarasovVarchenko2006, MukhinTarasovVarchenko2008,
  MukhinTarasovVarchenko2009, TarasovUvarov2020, Uvarov2022}.

As an application of our result, we obtain in Theorem \ref{t-2} and
Corollary \ref{c-1} a dynamical action of the {\em symmetric group} on
integrable representations  of $U_q\gl_M$, rather than of the braid
group as in \cite{EtingofVarchenko2002}. In fact our dynamical action
on functions with values in a representation $U$
differs from the dynamical action of the braid group defined in
\cite{EtingofVarchenko2002} by a transformation
$A_{g}(t)|_{U[\mu]}\mapsto A_{g}(t)|_{U[\mu]}f_{g,\mu}(t)$ for some
scalar meromorphic functions $f_{g,\mu}(t)$ obeying the cocycle
condition $f_{gh,\mu}(t)=f_{g,h\mu}(ht)f_{h,\mu}(t)$ for any braids
$g,h$, see Proposition \ref{p-3}. While in \cite{EtingofVarchenko2002}
the dynamical action of the braid group factors through the symmetric
group only when restricted to the zero weight space $U[0]$, our
modified action defines a representation of the symmetric group on the
whole space of $U$-valued functions.
In fact the same is true for general Kac--Moody algebras: our modified action
defines an action of the Weyl group on functions with values
in integrable representations, see Theorem \ref{t-KacMoody}. A proof
is provided by Anfisa Gurenkova in Appendix \ref{a-2}.

Part of the results of this paper were announced in the Oberwolfach report
\cite{Felder2022}.

\subsection*{Acknowledgments} 
This research was supported in part by
the National Centre of Competence in Research SwissMAP (grant numbers
182902, 205607) of the Swiss National Science Foundation. G. F.
was also supported by the grant 196892 of the Swiss National
Science Foundation.  This article was written in part while G. F.
was at IHES, which he thanks for hospitality.  A. G. was supported by the
RSF under project 23-11-00150.

We thank Tommaso Botta for various useful discussions and suggestions.
The second author thanks the organizers of the mini-workshop on
``Recent developments in representation theory and mathematical physics'' at the
Mathematisches Forschungsinstitut Oberwolfach, where these results were presented,
and the participants, expecially Evgeny Mukhin and Catharina Stroppel for discussions.
We thank Andrey Smirnov and Huafeng Zhang for useful correspondence.

\section{Braiding matrices for exterior powers of the vector representation}\label{s-2}
\subsection{\texorpdfstring{$q$}{q}-number notation}
We adopt the standard $q$-number notation: for $n\in\mathbb Z_{\geq0}$,
$[n]_q=\frac{q^n-q^{-n}}{q-q^{-1}}$. The $q$-factorial is
defined by $[0]_q!=1$, $[n]_q!=\prod_{j=1}^n[j]_q$, $n\geq 1$
and the $q$-binomial coefficient (or Gauss polynomial) is
\[
  \qbinomial{n}{j}=\frac{[n]_q!}{[j]_q![n-j]_q!}.
\]
All these objects belong to $\mathbb Z[q,q^{-1}]$.

\subsection{Quantum enveloping algebras of type
  \texorpdfstring{$A_{N-1}$}{A} and
 \texorpdfstring{ $A_{N-1}^{(1)}$}{A(1)}}
Drinfeld--Jimbo quantum enveloping algebras are the subject of several textbooks,
for example \cite{Lusztig1993, ChariPressley1994}. Here we collect our
conventions.
Let $q$ be a non-zero complex number which is not a root of unity.
Let $(a_{ij})_{i,j=0}^{N-1}$ be the Cartan matrix of $A_{N-1}^{(1)}$:
$a_{ij}=2$ if $i=j$, $a_{ij}=-1$ if $j=i\pm 1\mod N$, and $a_{ij}=0$
otherwise.  It is convenient to identify the index set $\{0,\dots,N-1\}$
with $\mathbb Z/N\mathbb Z$ in view of the symmetry $a_{i+1,j+1}=a_{ij}$.
The Hopf algebra $U'_q(\widehat{\mathfrak{sl}}_N)$ is the
unital algebra over $\mathbb C$ generated by $e_i,f_i,k_i^{\pm1}$,
$(i=0,\dots, N-1)$ and relations
\begin{align*}
  k_ik_j&=k_jk_i,\quad k_ik_i^{-1}=k_i^{-1}k_i=1
  \\
  k_ie_jk_i^{-1}
        &=q^{a_{ij}}e_j,\quad k_if_jk_i^{-1}=q^{-a_{ij}}f_j,
\\
  e_if_j-f_je_i
        &=\delta_{ij}\frac{k_i-k_i^{-1}}{q-q^{-1}}
  \\
  \sum_{r=0}^{1-a_{ij}} &\qbinomial {1-a_{ij}}r e_i^re_je_i^{1-a_{ij}-r}
  =0,\quad \text{if $i\neq j$,}
  \\
    \sum_{r=0}^{1-a_{ij}}& \qbinomial {1-a_{ij}}r f_i^rf_jf_i^{1-a_{ij}-r}
  =0,\quad \text{if $i\neq j$.}
\end{align*}
The coproduct is defined on generators as
$\Delta(e_i)=e_i\otimes 1+k_i\otimes e_i$,
$\Delta(f_i)=f_i\otimes k_i^{-1}+ 1\otimes f_i$ and
$\Delta(k^{\pm1}_i)=k^{\pm1}_i\otimes k^{\pm1}_i$.  The element
$\prod_{i=0}^{N-1}k_i$ is central and group-like. For finite
dimensional representations of $U_q'(\widehat{\mathfrak{sl}}_N)$ this
central element acts by 1. The quantum loop algebra
$U_qL{\mathfrak{sl}}_N$ is the quotient of $U'_q(\widehat{\mathfrak{sl}}_N)$
by the relation $\prod_{i=0}^{N-1}k_i=1$. The Hopf algebra structure
descends to this quotient. It is a deformation of the universal
enveloping algebra of the Lie algebra $L\mathfrak{sl}_N={\mathfrak{sl}}_N[t,t^{-1}]$ of
loops in ${\mathfrak{sl}}_N$.  The Hopf algebra
$U_qL{\mathfrak{sl}}_N$ contains the Hopf subalgebra
$U_q {\mathfrak{sl}}_N$ generated by $e_i,f_i,k_i$ with
$1\leq i\leq N-1$.

For $k\in\mathbb Z_{\geq0}$ the $k$th divided power of generators is
defined as
\[
  e_i^{(k)}=\frac{e_i^k}{[k]_q!},\quad f_i^{(k)}=\frac{f_i^k}{[k]_q!},
\]
with the convention that $e_i^{(0)}=f_i^{(0)}=1$.

Let $\mathbb Z^N=\oplus_{i\in\mathbb Z/N\mathbb Z}\mathbb Z\epsilon_i$
be the (co)weight lattice of $\gl_N$ with standard inner product
$(\ |\ )$, which identifies the Cartan subalgebra of diagonal matrix with its dual.
The root lattice of vectors of zero coordinate sum is
spanned by the simple roots $\alpha_i=\epsilon_i-\epsilon_{i+1}$,
$(i=1,\dots, N-1)$. We also set $\alpha_0=\epsilon_N-\epsilon_1$, the
highest root.  The Hopf algebra $U_qL\gl_N$ is obtained by adjoining
generators $t_\lambda$ for $\lambda\in \mathbb Z^N$ such that
$t_{\alpha_i}=k_i$ ($i=1,\dots,N-1$) with the relations
\[
  t_\lambda t_\mu=t_{\lambda+\mu},\quad t_0=1,\quad t_\lambda
  e_it_{-\lambda}=q^{(\lambda|\alpha_i)}e_i, \quad t_\lambda
  f_it_{-\lambda}=q^{-(\lambda|\alpha_i)}f_i.
\]
It is a Hopf algebra with
$\epsilon(t_\lambda)=1,
S(t_\lambda)=t_{-\lambda},\Delta(t_\lambda)=t_\lambda\otimes
t_\lambda$.  We set $t_i=t_{\epsilon_i}$ so that $k_i=t_it_{i+1}^{-1}$
for $i\in\mathbb Z/N\mathbb Z$. The Hopf subalgebra $U_q\gl_N$ is
generated by  $e_i,f_i$ with $1\leq i\leq N-1$ and all $t_i^{\pm1}$.

The multiplicative group $\mathbb C^\times $ acts on $U_qL\gl_N$ by
algebra automorphisms, deforming the rescaling action of the loop
parameter: the automorphism $\varphi_z$ associated with
$z\in\mathbb C^\times$ acts as the identity on all generators except
$e_0$, $f_0$, for which
\[
  \varphi_z(e_0)=ze_0,\quad \varphi_z(f_0)=z^{-1}f_0.
\]
The weight space of a $U_q\gl_N$-module of weight $\mu\in \mathbb Z^N$
is the simultaneous eigen\-space of $t_\lambda$ with eigenvalue
$q^{(\mu|\lambda)}$.  A $U_q\gl_N$-module is of {\em type I} if it is the direct
sum of its finite dimensional weight spaces.  All modules considered
in this paper are of type I.

\subsection{Exterior powers of the vector representation}
\label{ss-2.3}
Let $V=\mathbb C^N$ with basis $v_1,\dots,v_N$ and
$C_N$ be the Clifford algebra of $V\oplus V^*$ with quadratic form
$Q(v\oplus \alpha)=\alpha(v)$. It is the unital algebra  generated by
the images of the standard basis  and dual basis vectors
$\psi^*_i,\psi_i$, labeled by $i\in\{1,\dots,N\}$, which we often identify
with $\mathbb Z/N\mathbb Z$,  subject to the  relations
\begin{equation}\label{e-1}
  \psi_i\psi_j=-\psi_j\psi_i,\quad\psi_i^*\psi_j^*=-\psi_j^*\psi_i^*,\quad
  \psi_i\psi_j^*+\psi_j^*\psi_i=\delta_{ij}1.
\end{equation}
The Clifford algebra $C_N$ has a unique irreducible representation up
to isomorphism. It is $\bigwedge V$ with $\psi^*_i$ acting as
$v\mapsto v_i\wedge v$ and $\psi_i$ by contraction. It is generated by
$|0\rangle=1$ such that $\psi_i|0\rangle=0$ for all $i$. The commuting
``fermion number'' operators $n_i=\psi_i^*\psi_i$ have a common
eigenbasis $v_{i_1}\wedge\cdots\wedge v_{i_k}$ ($i_1<\dots<i_k$) and
eigenvalue $1$ if $i=i_j$ for some $j$ and $0$ otherwise. Thus
$t^{\pm1}_i:=q^{\pm n_i}=q^{\pm1}\psi_i^*\psi_i+\psi_i\psi_i^*$ obeys
the relations
\begin{equation}\label{e-2}
t_it_i^{-1}=1,\quad t_it_j=t_jt_i,\quad  t_i\psi_j^*t_i^{-1}=q^{\delta_{ij}}\psi_j^*,\quad
  t_i\psi_jt_i^{-1}=q^{-\delta_{ij}}\psi_j.
\end{equation}
for all $i,j$. Let $C_N(q)$ be the algebra with generators $t_i,\psi_i,\psi_i^*$,
($i=1,\dots,N$) and relations \eqref{e-1}, \eqref{e-2}. It is a $\mathbb Z/2\mathbb Z$-graded
algebra with $\psi,\psi_i^*$ odd and $t_i$ even, and $\bigwedge V$ is a $\mathbb Z/2\mathbb Z$-graded
module, with the convention that $|0\rangle$ is even.
\begin{prop}[Hayashi \cite{Hayashi1990}]\label{p-1}
  There is a homomorphism $U_qL\mathfrak{gl}_N
  \to C_N(q)$ such that
  \[
    e_i\mapsto \psi_i^*\psi_{i+1},\quad
    f_i\mapsto\psi_{i+1}^*\psi_i,\quad
    t_i^{\pm1}\mapsto t_i^{\pm 1},
  \]
  for $i\in\mathbb Z/N\mathbb Z$.
\end{prop}
Thus we have a representation of $U_qL\gl_N$ on $\bigwedge V$. It is
the direct sum of the irreducible eigenspaces $\bigwedge^kV$ of the
central element $\prod_{i=1}^N t_i$. These irreducible representations
are generated by the highest weight vectors
\begin{equation}\label{e-hw}
v^k=\psi_1^*\cdots\psi_k^*|0\rangle
\end{equation}
of weight $\sum_{i=1}^k\epsilon_i$.  In fact they are obtained from
the representation of the classical Lie algebra $L\gl_N$ on the same
representation space and with the same action of the generators
$e_i,f_i$ by setting $t_\lambda=q^{\lambda}$ for Cartan elements
$\lambda$.  The point is that, since the eigenvalues of $\epsilon_{i}$
are $1$ or $0$, the right-hand side $(q^{h_i}-q^{-h_i})/(q-q^{-1})$ of
the commutation relation of $e_i,f_i$ coincides with the action of its
classical counterpart $h_i=\epsilon_i-\epsilon_{i+1}$.

\begin{definition}\label{def-ext}
  The {\em $j$-th exterior power of the vector representation} of $U_qL\gl_N$ with spectral
  parameter  $z$ is the twist $\bigwedge^jV(z)$
  of $\bigwedge^jV$ by the automorphism $\varphi_z$. The vector representation is
  $V(z)=\bigwedge^1V(z)$.
\end{definition}
Thus the action of the generators is given by the formulae of
Proposition \ref{p-1} except that $e_0,f_0$ are mapped to
$z\psi_N^*\psi_{1}$ and $z^{-1}\psi_1^*\psi_N$, respectively.
\begin{remark} Our definition of the exterior power  representation
  has a slightly non-standard choice of signs. A more common convention is to
  let $e_0$ and $f_0$ act as $(-1)^{j-1}z\psi_N^*\psi_1$, $(-1)^{j-1}z^{-1}\psi_1^*\psi_N$, respectively,
  which in the limit $q\to 1$ correspond to the action of $zE_{N,1},z^{-1}E_{1,N}$, in terms of the
  standard basis $(E_{i,j})$ of $\gl_N$, so that $z$ has the
  interpretation of an ``evaluation point.'' Our convention is better suited to the fermionic representation
  and somewhat reduces the number of signs in the formulas. The relation of our definition of
  the exterior power $\bigwedge^jV(z)$ to the definition $V^{(j)}(z)$ in \cite{DateOkado1994}
  is $\bigwedge^jV(z)=V^{(j)}((-1)^{j-1}z)$,
\end{remark}

Another interpretation of the exterior powers, useful to compute $R$-matrices by the fusion procedure,
is in terms of $q$-deformed exterior products:
\begin{prop}[Jimbo \cite{Jimbo1986}]
  Let $k\geq2$ and $S$ be the subspace of $V^{\otimes k}$ spanned by
  $\alpha\otimes v_i\otimes v_i\otimes\beta$,
  $\alpha\otimes (v_i\otimes v_j+q v_j\otimes v_i)\otimes \beta$ for
  $i>j$, $\alpha\in V^{\otimes r},\beta\in V^{\otimes k-r-2}$,
  $r=0,\dots,k-2$. Then $S$ is a $U_qL\gl_N$ submodule of
  \[
    V^k(z):=V(q^{-k+1}z)\otimes\cdots\otimes V(q^{k-3}z)\otimes V(q^{k-1}z)
  \]
  and we have the isomorphism of $U_qL\gl_N$-modules
  $V^k(z)/S\to \bigwedge^k V(z)$ such that
  $v_{i_1}\otimes \cdots \otimes v_{i_k}+S\mapsto
  v_{i_1}\wedge\cdots\wedge v_{i_k}$ for $1\leq i_1<\cdots<i_k\leq N$.
\end{prop}

\subsection{Braiding matrices}
For generic $z_1,z_2$ the tensor products module
$\bigwedge^{k} V(z_1)\otimes \bigwedge^{k'} V(z_2)$ with action of
$U_q\widehat{\gl}_N$ given by the coproduct are irreducible and do not depend
on the order of factors up to isomorphism, {}see \cite[Proposition 12.2.15]{ChariPressley1994}. Thus there is an
isomorphism of $U_q\widehat{\gl}_N$-modules
\[
  \check R_{k,k'}(z_1/z_2)\colon\textstyle{\bigwedge^{k} V(z_1)\otimes
  \bigwedge^{k'} V(z_2)\to \bigwedge^{k'} V(z_2)\otimes
  \bigwedge^{k} V(z_1)},
\]
which, viewed as a linear map between the underlying vector spaces,
depends on the ratio of spectral parameters.  We find it convenient to
normalize it in such a way that $v^k\otimes v^{k'}$ is mapped to
$(-1)^{kk'}v^{k'}\otimes v^k$. Then the $R$-matrix
\[
  R_{k,k'}(z_1/z_2)=P\check R_{k,k'}(z_1/z_2)\in\operatorname{End}_{\mathbb C}(
  \textstyle{\bigwedge^{k} V(z_1)\otimes
  \bigwedge^{k'} V(z_2)}),
\]
obtained by the composition with the graded flip
$P\colon v\otimes w\to (-1)^{kk'}w\otimes v$, restricts to the identity
on the product of highest weight spaces.

  Then by the irreducibility of tensor products
at generic values of the spectral parameter we have the braiding relation
\begin{align}\label{e-3}
  (\check R_{k_2k_3}(z_2/z_3)
  &\otimes\mathrm{id})
    (\mathrm{id}\otimes\check R_{k_1k_3}(z_1/z_3))
    (\check R_{k_1k_2}(z_1/z_2)\otimes\mathrm{id})\notag
  \\
  &=
    (\mathrm{id}\otimes\check R_{k_1k_2}(z_1/z_2))
    (\check R_{k_1k_3}(z_1/z_3)\otimes\mathrm{id})
     (\mathrm{id}\otimes\check R_{k_2k_3}(z_2/z_3)),
\end{align}
on $
\bigwedge^{k_1}V(z_1)\otimes \bigwedge^{k_2}V(z_2)\otimes \bigwedge^{k_3}V(z_3)$,
equivalent to the Yang--Baxter equation for the $R$-matrices.
Moreover we have the inversion (or unitarity) relation
\begin{equation}\label{e-4}
  \check R_{k'k}(z^{-1})\check R_{kk'}(z)=\mathrm{id}.
\end{equation}
These $R$-matrices were computed
by Date and Okado \cite{DateOkado1994}, see also \cite{BrackenGouldZhangDelius1994},
in terms of projections onto irreducible
submodules for the action of $U_q\gl_N$. The coefficients were computed using
the fusion procedure \cite{KulishReshetikhinSklyanin1981}.
Here we give an alternative fermionic formula.

Let $\bigwedge V\otimes\bigwedge V$ be the tensor product of
$\mathbb Z/2\mathbb Z$-graded vector spaces.  The action of tensor
products of linear maps is given by the sign rule:
$(f\otimes g)(v\otimes w)=(-1)^{|g||v|}fv\otimes gw$ for $g$ of degree
$|g|$ and $v$ of degree $|v|$. We introduce the following even
endomorphisms of $\bigwedge V\otimes\bigwedge V$:
\begin{equation}\label{e-EFK}
      E=\sum_{i=1}^N
       (\psi^{*}_i\otimes {\psi}_i)\prod_{j>i}{K}_j
  \qquad
    F=-\sum_{i=1}^N
       \prod_{j<i}{K}^{{-1}}_j({\psi}_i\otimes\psi^{*}_i)
  \qquad
  {K}=\prod_{i=1}^N{{K}}_i,
\end{equation}
with  $K_i=t_i\otimes t_i^{-1}$ ($i=1,\dots,N$).
\begin{lemma}\label{l-1}
  The endomorphisms $E,F,K$ commute with the action of $U_q\gl_N$ and
  obey the commutation relations of $U_q\mathfrak{sl}_2$:
  \[
    [E,F]=\frac{K-K^{-1}}{q-q^{-1}},\quad KEK^{-1}=q^2E,\quad
    KFK^{-1}=q^{-2}F.
  \]
\end{lemma}
This is an instance of Howe duality and is discussed in the next
section. It is convenient to extend this action to an action of
$U_q\gl_2$ by setting $T_1=\prod_{i=1}^N t_i\otimes 1$,
$T_2=\prod_{i=1}^N1\otimes t_i$ so that
$\bigwedge^{k}V\otimes \bigwedge^{k'}V$ is the weight space of weight
$(k,k')$.

\begin{theorem}\label{t-1}
  The braiding matrix
  \[
    \check
  R_{k,k'}(z_1/z_2)\colon\textstyle{\bigwedge^kV(z_1)\otimes\bigwedge^{k'}V(z_2)\to
  \bigwedge^{k'}V(z_2)\otimes\bigwedge^{k}V(z_1)},
\]
normalized so that
  $\check R_{k,k'}(z_1/z_2)\,v^k\otimes v^{k'}=(-1)^{kk'}v^{k'}\otimes v^{k}$, is
  given by
  \[
    \check R_{k,k'}(z_1/z_2)=(-1)^{\min(k,k')}A_{k-k'}(z_1/z_2)
  \]
  where
  $A_{k-k'}(z_1/z_2)\in U_q\mathfrak{gl}_2$ is defined by
  \[
    A_m(z)=\sum_{j=0}^n
    (-q)^{j}\frac{1-q^{|m|}z}{1-q^{2j+|m|}z}
    \begin{cases}
      E^{(j)}F^{(j+m)},& \text{if $m\geq0$,}
      \\
      E^{(j-m)}F^{(j)},& \text{if $m\leq 0$,}
     \end{cases}
  \]
  with the divided power notation $E^{(l)}=E^l/[l]_q!$, $F^{(l)}=F^l/[l]_q!$ and
  \[
    n=\min(k,k',N-k,N-k').
  \]
\end{theorem}
The proof relies on Howe duality, which we explain in the next section,
and the computation
\cite{DateOkado1994} of the scalar by which $\check R$ acts on each irreducible
subrepresentation of the tensor product.

\subsection{Limit \texorpdfstring{$N\to\infty$}{N to infinity} and braiding matrices on Fock spaces}
Our formula for braiding matrices on  exterior powers of vector representations
does not depend explicitly on the dimension $N$ and have a formal
limit as $N\to\infty$ to braiding matrices acting on Fock spaces realized
following Dirac as semi-infinite forms $\mathcal F=\oplus_k\bigwedge^{\frac\infty2+k}\mathbb C^{\infty}$.
Here we show that the limit can be rigorously constructed so that our braiding
matrices converge to braiding matrices defined on Fock spaces.

Let $C_\infty$ be the infinite dimensional Clifford algebra generated
by $b_n$, $c_n$, $n\in\mathbb Z+\frac12$, with relations
$b_nc_m+c_mb_n=\delta_{n,-m}$. The Fock space $\mathcal F$ is the
module over $C_\infty$ generated by a vacuum vector $\vac$ annihilated
by $b_n$ and $c_n$ for $n>0$ and carries a representation of Lie
algebras of infinite matrices, see
e.g. \cite{KacRainaRozhkovskaya2013}, Lecture 4.  The Clifford algebra
has a $\mathbb Z$-grading so that $b_n$ has degree 1 and $c_n$ has
degree $-1$ for all $n$.  We have a compatible grading
\[
  \mathcal F=\oplus_{k\in\mathbb Z}\mathcal F_k.
\]
defined by giving degree $0$ to $\vac$.

We realize $\mathcal F_k$ as a direct limit
of $\mathcal F_k^{(N)}=\bigwedge^{N/2+k}\mathbb C^N$ as $N\to\infty$
with $N$ even.

Namely we renumber the standard basis of $\mathbb C^N$ as
$v_{-\frac{N-1}2},\dots,v_{\frac{N-3}2},v_{\frac{N-1}2}$ to identify $\mathbb C^N$
as a subspace of $\mathbb C^{N+2}$ for each $N$, and we embed
$i_N\colon \bigwedge \mathbb C^N\hookrightarrow \bigwedge \mathbb C^{N+2}$ as
\[
  i_N(\alpha)=\alpha\wedge v_{\frac{N+1}2},\quad N\in2\mathbb N.
\]
Then $i_N$ is a homomorphism of modules over the subalgebra
$C_{N}$ generated by $b_n,c_n$ for $|n|\leq (N-1)/2$, so
that $b_n$ is the exterior product with $v_n$ and $c_{-n}$ is the
contraction with the dual $v_n^*$. It is a {\em graded} homomorphism
if we assign degree $k$ to the subspace
$\mathcal F_k^{(N)}=\bigwedge^{\frac N2+k}\mathbb C^N$.  The relation
to the operators on $\bigwedge\mathbb C^N$ of Section \ref{ss-2.3} is
\[
  \psi^*_n=b_{-\frac {N+1}2+n},\quad \psi_n=c_{\frac{N+1}2-n},\quad
  (n=1,\dots,N) .
\]
Let 
\[
  \vac^{(N)}=v_{1/2}\wedge v_{3/2}\wedge\cdots\wedge v_{\frac{N-1}2}.
\]
Then $b_n\vac^{(N)}=0=c_n\vac^{(N)}$ for $n>0$ and $i_N$ maps
$\vac^{(N)}$ to $\vac^{(N+2)}$.  Thus {}$\mathcal F^{(N)}=\bigwedge \mathbb C^N$ is identified
with the subspace of {}$\mathcal F$ generated by $\vac$ over $C_N$ and we have 
a filtration of graded vector spaces
\[
  \cdots\subset \mathcal F^{(N)}\subset \mathcal F^{(N+2)}\subset\cdots,\quad N\in2\mathbb Z,
\]
with direct limit $\cup_{N}\mathcal F^{(N)}=\mathcal F$. In particular,
\[
  \mathcal F_k=\cup_N\mathcal F_k^{(N)}
\]
The renormalized fermion number $\tilde n_i\in C_N$ is
\[
  \tilde n_i=\,\colon b_ic_{-i}\colon=
  \begin{cases} b_ic_{-i}& \text{if $i<0$,}
    \\
    -c_{-i}b_i& \text{if $i>0$.}
  \end{cases}
\]
It commutes with $i_N$ and is normalized so that $\tilde n_i\vac^{(N)}=0$. Then
$\tilde t_i=q^{\tilde n_i}$ is well-defined on the direct limit. It has the properties
\begin{equation}\label{e-tt}
  \tilde t_n\vac=1,\quad
  \tilde t_nb_m\tilde t_n^{-1}=q^{\delta_{n,m}}b_m,\quad \tilde
  t_nc_{-m}\tilde t_n^{-1}=q^{-\delta_{n,-m}}c_{-m},
\end{equation}
for all $n,m\in\mathbb Z+\frac12$.
The restriction of $\tilde t_n$ to $\mathcal F^{(N)}$ is
related to the generator $t_n$ of Section \ref{ss-2.3} by
$\tilde t_n=q^{-N/2}t_{n+(N-1)/2}$.

The representation of $U_q\mathfrak{sl}_2$ on
$\mathcal F^{(N)}\otimes \mathcal F^{(N)}$ obtained by Howe duality
defined by \eqref{e-EFK} is then given in this new notation by the following
action of generators (where we add the dependence on $N$ in the notation):
\begin{align*}
  E(N)&=\sum_{i=-\frac{N-1}2}^{\frac{N-1}2} (b_i\otimes
  c_{-i})\prod_{j=i+1}^{\frac{N-1}2}\tilde K_j,
\quad
  F(N)=-\sum_{i=-\frac{N-1}2}^{\frac{N-1}2} \prod_{j=-\frac{N-1}2}^{i-1}\tilde
  K_j(c_{-i}\otimes b_i)
  \\
  K(N)&=\prod_{i=-\frac{N-1}2}^{\frac{N-1}2}\tilde K_i,\quad \tilde
  K_i=\tilde t_i\otimes \tilde t_i^{-1}
\end{align*}
\begin{prop}
  Let $X(N)=E(N),F(N),$ or $K(N)$ and $N\in2\mathbb Z$. Then
  \[
    (i_N\otimes i_N)\circ X(N)=X(N+2)\circ (i_N\otimes i_N).
  \]
  Thus $E(N)$, $F(N)$, $K(N)$ are the restrictions to the invariant
  subspace
  $\mathcal F^{(N)}\otimes\mathcal F^{(N)}\subset \mathcal F\otimes
  \mathcal F$ of well-defined operators
  $E,F,K\in\operatorname{End}(\mathcal F\otimes \mathcal F)$ defining
  a representation of $U_q\mathfrak{sl}_2$.
\end{prop}

\begin{proof} Since $\mathcal F^{(N)}$ is generated by $\vac^{(N)}$ as a
  module over the Clifford algebra $C_N$, $\tilde n_i$ is uniquely
  characterized by the commutation relations \eqref{e-tt} and the
  condition that $\tilde n_i\vac^{(N)}=0$.  Thus for $N<N'$ and
  $|i|>(N-1)/2$, $\tilde t_i$ acts as $1$ on
  $\mathcal F^{(N)}\subset \mathcal F^{(N')}$.  In particular for
  $|i|=(N+1)/2$, $\tilde K_i$ acts by 1 on the image of
  $\mathcal F^{(N)}$ in $\mathcal F^{(N+2)}$ and thus
  $K(N+2)\circ (i_N\otimes i_N)=(i_N\otimes i_N)\circ K(N)$. As for
  $E(N)$, we have similarly
  \[
    (i_N\otimes i_N)\circ\prod_{j=i+1}^{\frac{N-1}2}\tilde
    K_j=\prod_{j=i+1}^{\frac{N+1}2}\tilde K_j\circ (i_N\otimes i_N),
  \]
  since the last factor on the right-hand side acts by 1. Also
  $   b_i\otimes c_{-i}$ commutes with $i_N\otimes i_N$ for $|i|\leq \frac{N-1}2$
  and acts by zero on the image of $i_N\otimes i_N$
  if $|i|=(N+1)/2$.
  This is because for $|i|>(N-1)/2$, $b_i$ and $c_{-i}$ commute
  (up to sign) with $C_N$ and either $c_{-i}\vac^{({N+2})}$ or
  $b_i\vac^{({N+2})}$ vanishes.
  Thus $E(N+2)\circ i_N\otimes i_N=i_N\otimes i_N\circ E(N)$ and similarly
  for $F(N)$.
\end{proof}

We can formally write the limiting representation of $U_q\mathfrak{sl}_2$ on
$\mathcal F\otimes \mathcal F$ as
\begin{align*}
  E&=\sum_{i=-\infty}^{\infty} (b_i\otimes
  c_{-i})\prod_{j=i+1}^{\infty}\tilde K_j,
\quad
  F=-\sum_{i=-\infty}^{\infty} \prod_{j=-\infty}^{i-1}\tilde
  K_j(c_{-i}\otimes b_i)
  \\
  K&=\prod_{i=-\infty}^{\infty}\tilde K_i,\quad \tilde
  K_i=\tilde t_i\otimes \tilde t_i^{-1}
\end{align*}
with the remark that, when acting with these operators on any element of
$\mathcal F\otimes\mathcal F$, all but finitely many summands in the infinite
sums vanish and all but finitely many factors in the products are equal to 1.

\begin{theorem}\label{t-Fock}
  Let $E^{(j)}=E/[j]_q!$, $F^{(j)}=F/[j]_q!$, $(j=0,1,\dots)$.
  For every $m\in \mathbb Z$ the operator
  \[
    A_m(z)=\sum_{j=0}^\infty
    (-q)^{j}\frac{1-q^{|m|}z}{1-q^{2j+|m|}z}
    \begin{cases}
      E^{(j)}F^{(j+m)},& \text{if $m\geq0$,}
      \\
      E^{(j-m)}F^{(j)},& \text{if $m\leq 0$,}
     \end{cases}
   \]
   is a well-defined rational function with values in $\operatorname{End}(\mathcal F
   \otimes\mathcal F)$.
  The braiding matrix
  \[
    \check
    R_{k,k'}(z)= (-1)^{\min(k,k')}A_{k-k'}(z)
  \]
  maps $\mathcal F_k\otimes \mathcal F_{k'}$ to $\mathcal F_{k'}\otimes \mathcal F_k$
  and is a solution of the Yang--Baxter equation \eqref{e-3} and the inversion relation
  \eqref{e-4}.
\end{theorem}

\begin{proof}
  The operators $E$ and $F$ preserve the finite dimensional subspaces
  $\mathcal F^{(N)}\otimes\mathcal F^{(N)}$ and they are nilpotent there
  so that the sums defining $A_m$ reduce to finite sums. Moreover
  $\check R_{k,k'}(z)$ restricts to the braiding matrix of Theorem \ref{t-1}
  on these subspaces, so the Yang--Baxter equation and the inversion relation
  follow.
\end{proof}

.
\section{Quantum skew Howe duality}\label{s-3}
The classical skew-symmetric version of Howe duality \cite[Section 4]{Howe1992}
states that the natural actions of $\mathit{GL}_N$ and
$\mathit{GL}_M$ on $\bigwedge(\mathbb C^N\otimes\mathbb C^M)$ generate
each other's commutant and one has a decomposition into irreducibles
of $\mathit{GL}_N\times \mathit{GL}_M$
\[
  \bigwedge (\mathbb C^N\otimes\mathbb C^M)\cong
  \bigoplus_{\lambda}V^N_\lambda\otimes V^M_{\lambda^t}.
\]
The sum is over all Young diagrams $\lambda$ with at most $N$ rows
such that the transposed diagrams $\lambda^t$ has at most $M$ rows
(equivalently, the length of the first row of $\lambda$ is at most
$M$) and $V^N_\lambda$ denotes the irreducible representation of
$\mathit{GL}_N$ of highest weight $\lambda$. The following nice
pictorial description of the corresponding highest weight vectors is
due to Howe in {\it loc.~cit.}: write the tensor products of basis vectors
$v_{i,r}=v_i\otimes w_r$ $(i=1,\dots,N, r=1,\dots, M)$ in the boxes of an
$N\times M$ grid (the tensors $v_{1,r}$ go in the first row,
$v_{2,r}$ in the second row, etc.).   The Young diagrams
occuring in the decomposition fit in this rectangle. The highest
weight vector corresponding to $\lambda$ is obtained by taking the
wedge product of the basis vectors in the boxes of the Young diagram
$\lambda$. For example, if $\lambda=(3,1)$,
we take $v_{1,1}\wedge v_{1,2}\wedge v_{1,2}\wedge v_{2,1}$.
It is indeed of weight $\lambda$ for $\mathit{GL}_N$ and of
weight $\lambda^t$ for $\mathit{GL}_M$.

The quantum group version of Howe duality is known: in the analogous
case of the symmetric algebra it is discussed in \cite{Baumann1999},
\cite{ToledanoLaredo2002} and \cite{Zhang2002}.
The skew-symmetric case we use here is
explained in \cite{CautisKamnitzerLicata2010} (for $M=2$), \cite{LehrerZhangZhang2011} and
\cite{CautisKamnitzerMorrison2014}.  
Here we present
an equivalent variant based on the Clifford algebra action.

Let $V=\mathbb C^N$ with basis $v_1,\dots,v_N$ be the vector
representation of $U_q\mathfrak {gl}_N$ and $W=\mathbb C^M$ with basis
$w_1,\dots,w_M$ be the vector representation of
$U_q\mathfrak{gl}_M$. Then we have an isomorphism of graded
commutative algebras
\[
  \bigwedge(V\otimes W)\cong\bigwedge(V\oplus\cdots\oplus V) \to
  \left(\bigwedge V\right)^{\otimes M} =\bigwedge
  V\otimes\cdots\otimes \bigwedge V
\]
sending $v_i\otimes w_j$ to
$1\otimes\cdots\otimes 1\otimes v_i\otimes 1\otimes \cdots\otimes 1$
with $v_i$ placed in the $j$-th factor. The right-hand side is a
tensor product of $U_q\mathfrak{gl}_N$-modules and thus
$U_q\mathfrak{gl}_N$ acts on it by the iterated coproduct. Similarly
we have an isomorphism
$\bigwedge(V\otimes W)\cong \bigwedge W\otimes\cdots \otimes\bigwedge
W$ sending $v_i\otimes w_j$ to
$1\otimes \cdots\otimes 1\otimes w_j\otimes 1\otimes \cdots \otimes 1$
with $i$-th factor $w_j$.  Here we let $U_q\mathfrak{gl}_M$ act via
the {\em opposite} coproduct $\Delta'=\sigma\circ\Delta$, with
$\sigma(a\otimes b)=b\otimes a$.

\begin{prop}\label{p-5}
  The actions of $(U_q\mathfrak{gl}_N,\Delta)$ and
  $(U_q\mathfrak{gl}_M,\Delta')$ on $\bigwedge(V\otimes W)$ pulled
  back by the isomorphisms
  \[
    \left(\bigwedge W\right)^{\otimes N}\leftarrow\bigwedge(V\otimes W)
    \rightarrow\left(\bigwedge V\right)^{\otimes M}
  \]
  commute. If $q$ is not a root of unity the images of
  $U_q\mathfrak{gl}_N$ and $U_q\mathfrak {gl}_M$ in
  $\operatorname{End}\bigwedge(V\otimes W)$ are commutants of each
  other and we have a decomposition
  \begin{equation}\label{e-dec}
    \bigwedge(V\otimes W)\cong\bigoplus_\lambda V^N_\lambda\otimes V^M_{\lambda^t}
  \end{equation}
  into simple $U_q\mathfrak{gl}_N\otimes U_q\mathfrak{gl}_M$-modules,
  where $\lambda$ runs over Young diagrams with at most $N$ rows and
  at most $M$ columns.
\end{prop}

For the proof we view $\bigwedge(V\otimes W)$ as a representation of
the Clifford algebra of $V\otimes W\oplus (V\otimes W)^*$. This algebra has
generators $\psi^*_{i,r},\psi_{i,r}$ corresponding to the tensor basis
$v_i\otimes w_r$ and its dual basis. As in Section \ref{ss-2.3} we introduce the
operators
$t_{i,r}^{\pm1}=q^{\pm1}\psi_{i,r}^*\psi_{i,r}+\psi_{i,r}\psi_{i,r}^*$.
Via the isomorphism to $(\bigwedge V)^{\otimes M}$, $\psi_{i,r}^*$
acts as
$\mathrm{id}^{\otimes(r-1)}\otimes \psi_i^*\otimes\mathrm{id}^{\otimes
  (N-r)}$ and $\psi_{i,r}$ acts as
$\mathrm{id}^{\otimes(r-1)}\otimes \psi_i\otimes\mathrm{id}^{\otimes
  (N-r)}$. Thus the pull-back of the action of $U_q\mathfrak{gl}_N$ on
$\bigwedge(V\otimes W)$ is given on generators by
\begin{align*}
e_i&\mapsto\sum_{r=1}^M\prod_{s=1}^{r-1} k_{i,s}
\psi^*_{i,r}\psi_{i+1,r}, \quad k_{i,s}=t_{i,s}t_{i+1,s}^{-1}\\
f_i&\mapsto\sum_{r=1}^M
     \psi^*_{i+1,r}\psi_{i,r}\prod_{s={r+1}}^M k_{i,s}^{-1}
  \\
t_i^{\pm1}&\mapsto \prod_{r=1}^Mt_{i,r}^{\pm1}.
\end{align*}
Similary we can pull-back the action of $U_q\mathfrak{gl}_M$ on
$(\bigwedge W)^{\otimes N}$ with the opposite coproduct. The action
of generators (denoted by capital letters to distinguish them)
\begin{align*}
  E_r&\mapsto\sum_{i=1}^N\psi^*_{i,r}\psi_{i,r+1}
       \prod_{j=i+1}^{N} K_{j,r},\quad K_{j,r}=t_{j,r}t_{j,r+1}^{-1},
  \\
F_r&\mapsto\sum_{i=1}^N
\prod_{j=1}^{i-1} K_{j,r}^{-1}\psi^*_{i,r+1}\psi_{i,r},
\\
T_r^{\pm1}&\mapsto \prod_{i=1}^Nt_{i,r}^{\pm1}.
\end{align*}
In the case $M=2$ this reduces to the formulae of the previous section, by remembering the sign rule
in the definition of $F$. By construction, we obtain actions of $U_q\gl_N$ and $U_q\gl_M$ and we need
to check that they commute:
\begin{lemma}\label{l-2}
  The actions of $U_q\gl_N$ and $U_q\gl_M$ on $\bigwedge(V\otimes W)$
  commute.
\end{lemma}
\begin{proof} It is clear that $t_i^{\pm1}$ commutes with the action
  of $U_q\gl_M$ and that $T_r^{\pm1}$ commutes with the action of
  $U_q\gl_N$.  Let us check that $[e_i,E_r]=0$. The cases of the other
  generators are dealt with in the same way.
  The only terms in the sum
  defining $e_i$ and $E_r$ that contribute non-trivially to the commutator  are
  \begin{align*}
    [e_i,E_r]&=\prod_{s<r}k_{i,s}  [
               \psi_{i,r}^*\psi_{i+1,r}+k_{i,r}\psi_{i,r+1}^*\psi_{i+1,r+1},
               \\
    &\quad
      \psi_{i,r}^*\psi_{i,r+1}K_{i+1,r}+
      \psi_{i+1,r}^*\psi_{i+1,r+1}] \prod_{j>i+1}K_{j,r}.
  \end{align*}
  The bracket on the right-hand side is
  \begin{align*}
    &[\psi_{i,r}^*\psi_{i+1,r},      \psi_{i+1,r}^*\psi_{i+1,r+1}]
    +[k_{i,r}\psi_{i,r+1}^*\psi_{i+1,r+1},      \psi_{i,r}^*\psi_{i,r+1}K_{i+1,r}]
    \\
    &=\psi_{i,r}^*\psi_{i+1,r+1}
      +k_{i,r}\psi_{i+1,r+1}\psi_{i,r}^*K_{i+1,r}
    \\
    &=\psi_{i,r}^*\psi_{i+1,r+1}+t_{i,r}\psi_{i+1,r+1}\psi_{i,r}^*t_{i+1,r+1}^{-1}
  \end{align*}
  We have the relations $t_{i,r}\psi_{i,r}^*=q\psi_{i,r}^*$ and
  $\psi_{i+1,r+1}t_{i+1,r+1}^{-1}=q^{-1}\psi_{i+1,r+1}$ and we are left
  with $\psi_{i,r}^*\psi_{i+1,r+1}+\psi_{i+1,r+1}\psi_{i,r}^*=0$.
\end{proof}
To complete the proof of Proposition \ref{p-5}, we observe that for
$q$ not a root of unity the decomposition into irreducible
representations in the classical case deforms to the quantum group
case.  In fact the highest weight vectors generating irreducible
$U_q\gl_N\otimes U_q\gl_M$-modules are still given by Howe's
construction.

\begin{proof}[Proof of Theorem \ref{t-1}]
  We use the formula of Date and Okado \cite{DateOkado1994} for the action of the $R$-matrix on highest weight vectors
  of irreducible $U_q\gl_N$-subrepresentations in
  $\bigwedge^kV\otimes \bigwedge^{k'}V$. The decomposition into
  irreducibles is best understood in terms of Howe duality: we consider
  the case $M=2$ of the decomposition \eqref{e-dec}
  of Proposition \ref{p-5}. The left-hand side is
  \[
    \textstyle{
    \bigwedge(V\otimes\mathbb C^2)=\bigwedge(V\oplus V)=
    \bigwedge V\otimes\bigwedge V=\oplus_{k,k'=0}^N\bigwedge^kV\otimes
    \bigwedge^{k'} V
    }
  \]
  and the summand labeled by $(k,k')\in\{0,\dots,N\}^2$ is the weight
  space of weight $(k,k')$ for the $U_q\gl_2$-action. The right-hand
  side is a sum over Young diagram $\lambda$ with at most two columns
  and at most $N$ rows.  Thus
  $\lambda=(2,\dots,2,1,\dots,1)=2^{\ell'}1^{\ell-\ell'}$ ($\ell'$
  rows of length 2, $\ell-\ell'$ rows of length 1) and
  $\lambda^t=(\ell,\ell')$ with $N\geq \ell\geq\ell'\geq0$.  The
  irreducible $U_q\gl_2$-module $V^2_{(\ell,\ell')}$ has a non-trivial
  weight space of weight $(k,k')$ if and only if $k+k'=\ell+\ell'$
  for some $s\geq0$ and $\ell'-\ell\leq k-k'\leq \ell-\ell'$.
  In this case the weight space is one-dimensional.
  If $k\geq k'$ the corresponding highest weights are
  \[
    (\ell,\ell')=(k,k'),(k+1,k'-1),\dots,(k+n,k'-n),
  \]
  where $n=\min(k',N-k)$, whereas if $k\leq k'$ they are
  \[
    (\ell,\ell')=(k',k),(k'+1,k-1),\ldots,(k'+n,k-n),
  \]
  where $n=\min(k,N-k')$. In both cases we can write
  \[
    n=\min(k,k',N-k,N-k').
  \]
  With this notation, the weight $(k,k')$ subspace in the decomposition \eqref{e-dec}
  gives the decomposition
  \[
  {\textstyle  \bigwedge^{k}V\otimes\bigwedge^{k'}V}\cong\oplus_{s=0}^{n}
    V^N_{(\overline k+s,\underline k-s)^t},\quad
    \text{with $\overline k=\max(k,k')$ and
      $\underline k=\min(k,k')$},
  \]
  as an $U_q\gl_N$-module. The $s=0$ component is generated by the
  highest weight vector $v^{(\overline{k},\underline{k})}_0=v^{\overline{k}}
  \otimes v^{\underline{k}}$, see \eqref{e-hw},  and
  the highest weight vector $v^{(k,k')}_s$ of $V^N_{(\overline k+s,\underline k-s)^t}$
  is proportional to $F^{(s)}v_0^{(k+s,k'-s)}$ if $k\geq k'$ and to
  $F^{(k'-k+s)}v_0^{(k'+s,k-s)}$ if $k\leq k'$.
  We normalize it as in \cite{DateOkado1994}
 by fixing the coefficient of
  the basis vector $\psi^*_k\cdots\psi^*_1|0\rangle$ in the first factor. Namely,
  \[
    v_s^{(k,k')}=\psi_k^*\cdots\psi_1^*|0\rangle\otimes\psi^*_{\overline
      k+s}\cdots \psi^*_{\overline k+1}\psi^*_{\underline
      k-s}\cdots\psi_1^*|0\rangle+\cdots.
  \]
  Let
  \[
    \textstyle Q_s\colon \bigwedge^k
    V\otimes\bigwedge^{k'}V\to\bigwedge^{k'} V\otimes\bigwedge^{k}V
  \]
  be the unique $U_q\gl_N$-linear map sending $v^{(k,k')}_s$ to the
  vector of the same weight $v^{(k',k)}_{s'}$, $s'=k-k'+s$ and vanishing on the
  other irreducible components. In our
  notation and with the normalization of Theorem \ref{t-1}, the formula of
  \cite{DateOkado1994}, obtained by the fusion procedure, is
  \begin{equation}\label{e-DateOkado}
    \check R_{k,k'}(z)
    =\sum_{s=0}^{n}(-1)^{kk'+(k-k')s}\prod_{j=1}^s\frac{z-q^{|k-k'|+2j}}{1-z
      q^{|k-k'|+2j}}Q_s.
  \end{equation}
  To deduce a formula in terms of the $U_q\gl_2$-action we need
  to compare the normalizations of the highest weight vectors.
  If $k\geq k'$ a straightforward computation using the formula for $F$
  in \eqref{e-EFK}  shows that
  \[
    v^{(k,k')}_s=(-1)^{sk}q^{ss'}F^{(s)}v^{(k+s,k'-s)}_0,
\quad
    v^{(k',k)}_{s'}=(-1)^{s'k'}q^{ss'}F^{(s')}v^{(k+s,k'-s)}_0.
  \]
  It follows that for $k\geq k'$
  \[
    \check R_{k,k'}(z)F^{(s)}v_0^{(k+s,k'-s)}
    =(-1)^{k'}\prod_{j=1}^s\frac{z-q^{|k-k'|+2j}}{1-z
      q^{|k-k'|+2j}}F^{(s')}v_0^{(k+s,k'-s)}
  \]
  To prove the formula for $\check R$ given in the Theorem
  we need to compute the action of $A_{k-k'}(z)$ on the irreducible
  representation generated by the highest weight vector
  $v^{(k+s,k'-s)}_0$ of $\mathfrak{sl}_2$-weight
  $\ell=k-k'+2s\in\mathbb Z_{\geq0}$.  Comparing with the formula for
  $A_m(z)$ from Appendix \ref{a-1}  (see Lemma \ref{l-A1}) we see that
  $\check R_{k,k'}(z)$ acts by the same factor on $F^{(s)}v_0^{(k+s,k'-s)}$
  up to the sign $(-1)^{k'}=(-1)^{\min(k,k')}$.

  If $k\leq k'$ the reasoning is similar: in this case the
  highest weight vector $v^{(k,k')}_s$ is proportional
  to $F^{(s')}v_0^{(k'+s,k-s)}$ with $s'=k'-k+s$ and
  \[
    v^{(k,k')}_s=(-1)^{sk'}q^{ss'}F^{(s')}v^{(k'+s,k-s)}_0, \quad
    v^{(k',k)}_{s'}=(-1)^{s'k}q^{ss'}F^{(s)}v^{(k'+s,k-s)}_0.
  \]
  From this we obtain
  \[
    \check R_{k,k'}(z)F^{(s')}v_0^{(k'+s,k-s)}
    =(-1)^{k}\prod_{j=1}^s\frac{z-q^{|k-k'|+2j}}{1-z
      q^{|k-k'|+2j}}F^{(s)}v_0^{(k'+s,k-s)},
  \]
  which is the same as the action of $A_{k-k'}(z)$ up to the sign
  $(-1)^{k}=(-1)^{\min(k,k')}$.
\end{proof}

\section{Dynamical Weyl group}\label{s-4}
\subsection{Braiding matrices as generators of the Weyl group}
In Theorem \ref{t-1} the braiding matrix for the tensor product of two
representations of $U_qL\mathfrak{gl}_N$ obtained from exterior powers
of the vector representation of $U_q\mathfrak{gl}_N$ are expressed
in terms of the action of $U_q\mathfrak{gl}_2$ via Howe duality.
In this section we consider braiding matrices on the tensor product of $M$
exterior powers and notice that they are correspondingly expressed
in terms of the action of $U_q\mathfrak{gl}_M$ via the simple root embeddings of
$U_q\gl_2$.
The braid and inversion relations of the $\check R$-matrices acting on
a $M$-fold tensor product  is then translated by Howe duality to braid and inversion
relations in $U_q\gl_M$.
The result is then that these are relations of the
dynamical Weyl group.

Recall that the (ordinary) Weyl group of $\mathit{GL}_M$ is
the symmetric group $S_M$. It acts on $\mathbb C^M$ and
the weight lattice $\mathbb Z^M$ by permutations. It is generated by the transpositions $s_i\in S_M$ of
$i$ and $i+1$ ($i=1,\dots,M-1$) with the relations $s_i^2=1$, $s_is_j=s_js_i$ for
$|i-j|\geq 2$ and $s_is_{i+1}s_i=s_{i+1}s_is_{i+1}$ for
$1\leq i\leq M-2$; the dynamical Weyl group relations are a version
with parameters.

Let us define formal series $A_{i,\mu}(z)$ of rational functions of
$M$-variables $z=(z_1,\dots,z_M)$ with values in
$U_q{\mathfrak{sl}}_M\subset U_q\gl_M$ for $1\leq i\leq M-1$ and
$\mu\in P_M=\mathbb Z^M$:
\[
  A_{i,\mu}(z)=\sum_{j=0}^{\infty}
  (-q)^{j}
  \frac{z_{i+1}-q^{|\mu_i-\mu_{i+1}|}z_i}
  {z_{i+1}-q^{2j+|\mu_i-\mu_{i+1}|}z_i}
  E_i^{(j+(\mu_{i+1}-\mu_i)_+)}F_i^{(j+(\mu_i-\mu_{i+1})_+)}.
\]
Here we set $(x)_+=\max(x,0)$ and, as before, we denote the generators
by capital letters to distinguish them from the generators of
$U_q\gl_N$. For any finite dimensional $U_q\gl_N$-module $U$ of type I
only finitely many terms contribute non-trivially to the sum and
$A_{i,\mu}(z)$ maps the weight space $U[\mu]$ of weight $\mu$ to
$U[s_i\mu]$.  Notice that $A_{i,\mu}(z)$ depends on $\mu$ through its class
in the weight lattice $P_M=\mathbb Z^M/\mathbb Z(1,\dots,1)$ of $\mathfrak{sl}_M$.
For $M=2$, the map $\mathbb Z^2\to\mathbb Z$, $\mu\mapsto \mu_1-\mu_2$ induces
a bijection $P_2\cong \mathbb Z$. With this identification, $A_{1,m}(z_1,z_2)=A_m(z_1/z_2)$ where
\begin{equation}\label{e-5}
  A_m(z)= \sum_{j\geq0} (-q)^{j}\frac{1-q^{|m|}z}{1-q^{2j+|m|}z}
  \begin{cases}
    E^{(j)}F^{(j+m)},& \text{if $m\geq 0$,}
    \\
    E^{(j+|m|)}F^{(j)},& \text{if $m\leq 0$.}  \end{cases}
\end{equation}
The general series $A_{i,\mu}(z)$ reduces to this case via the
embedding $j_i\colon U_q\mathfrak{sl}_2\hookrightarrow
U_q\mathfrak{sl}_M$ sending $E,F,K^{\pm1}$ to $E_i,F_i,K_i^{\pm1}$:
\[ A_{i,\mu}(z)=j_i\left(A_{\mu(h_i)}(z_i/z_{i+1})\right).
\] By construction, see Theorem \ref{t-1}, the action of
$A_{i,\mu}(z)$ on the tensor product of exterior powers
$\bigwedge^{\mu_i}V(z_i)$ via Howe duality is the action of the
braiding matrix $\check R_{\mu_i,\mu_{i+1}}(z_i/z_{i+1})$ on the $i$th
and $i+1$st factors (up to an irrelevant sign). Note that this tensor
product is the weight space of weight $\mu$ for the action of
$U_q\gl_M$ on the tensor product of $M$ copies of $\bigwedge V$.

\begin{theorem}\label{t-2} Let $U$ be any finite dimensional
$U_q\gl_M$-module of type I. Then the restriction of $A_{i,\mu}(z)$ to
the weight space $U[\mu]$ of weight $\mu$ obeys the following
relations.
  \begin{enumerate}
   \item[(i)] For $1\leq i,j\leq M-1$ such that $|i-j|\geq 2$,
     \[ A_{i,s_i\mu}(s_jz)A_{j,\mu}(z) =
A_{j,s_j\mu}(s_iz)A_{i,\mu}(z),
     \] in $\operatorname{Hom}(U[\mu],U[s_is_j\mu])$.
    \item[(ii)] For all $1\leq i\leq M-1$,
      \[ A_{i,s_i\mu}(s_iz)A_{i,\mu}(z) =\mathrm{id}
      \] in $\operatorname{End}(U[\mu])$.
   \item[(iii)] For $1\leq i\leq M-2$,
     \begin{align*} A_{i,s_{i+1}s_i\mu}&(s_{i+1}s_iz)
A_{i+1,s_i\mu}(s_iz)A_{i,\mu}(z) \\ &=
A_{i+1,s_is_{i+1}\mu}(s_is_{i+1}z)
A_{i,s_{i+1}\mu}(s_{i+1}z)A_{i+1,\mu}(z),
     \end{align*} in $\operatorname{Hom}(U[\mu],U[s_is_{i+1}s_i\mu])$.
   \end{enumerate}
 \end{theorem}

 \begin{proof} By complete reducibility we may assume that $U$ is
irreducible.  The relations hold for the action of $A_{i,\mu}$ on
$\otimes_{i=1}^M\bigwedge^{\mu_i}V(z_i)$ via Howe duality: (i) follows
from the fact that the corresponding $\check R$-matrices act on
different factors of the tensor product, (ii) is the inversion
relation \eqref{e-4} and (iii) is the braiding relation \eqref{e-3}.
Thus the relations hold for any irreducible $U_q\gl_M$-module
occurring in the decomposition of the tensor product. But from
Proposition \ref{p-5} we see that any  partition does occurs as a highest weight if we
take $N$ sufficiently large. This covers all type I irreducible
representations of $U_q\mathfrak{sl}_M$. For $U_q\gl_M$ we can obtain
all representations by tensoring these with a power of the one
dimensional representations on which $U_q\mathfrak{sl}_M$ acts
trivially (i.e. by the counit) and $t_i$ act by $q^{-1}$ for all
$i$. This has the effect of shifting the $\mu_i$s by a common amount
and does not affect the validity of the claim, since $A_{i,\mu}(z)$
depends on differences of $\mu_i$s.
 \end{proof} This result can be reformulated as the construction of a
representation of the Weyl group on $U$-valued functions. Let $U$ be a
finite-dimensional type I $U_q\gl_M$-module with weight decomposition
$U=\oplus_\mu U[\mu]$. Let $A_{i,U}(z)$ be the
$\operatorname{End}U$-valued rational function such that
\begin{equation}\label{e-6}
  A_{i,U}(z)|_{U[\mu]}=A_{i,\mu}(z)\colon U[\mu]\to U[s_i\mu]
\end{equation}
for all weights $\mu$.

\begin{cor}\label{c-1} Let $U$ be a finite-dimensional $U_q\gl_M$-module of type I.
The endomorphisms $s_1,\dots, s_{M-1}$ of the vector space $U(z)$ of
$U$-valued rational functions in $M$ variables $z$ given by
  \[ (s_if)(z)=A_{i,U}(s_iz)f(s_iz),\quad i=1,\dots,M-1,
  \] define a representation of the symmetric group $S_M$.
\end{cor}
\begin{proof} The proof is a straightforward verification that the
relations of the symmetric group are equivalent to the identities
(i)--(iii) of Theorem \ref{t-2}. For example, to check the
relation $s_is_{i+1}s_i=s_{i+1}s_is_{i+1}$ we write
  \begin{align*} (s_is_{i+1}s_if)(z) &=A_{i,U}(s_iz)(s_{i+1}s_if)(s_iz)
\\ &=A_{i,U}(s_iz)A_{i+1,U}(s_{i+1}s_iz)(s_if)(s_{i+1}s_iz) \\
&=A_{i,U}(s_iz)A_{i+1,U}(s_{i+1}s_iz)A_{i,U}(s_is_{i+1}s_iz)f(s_{i}s_{i+1}s_iz).
  \end{align*}
  We get the left-hand side of (iii) if we replace $z$ by $ s_is_{i+1}s_iz$ and
  restrict to the weight space $U[\mu]$.
\end{proof}

We notice that the operators $A_{i,U}(z)$ have similar properties as
the operators of the dynamical Weyl group of $U$ introduced and
studied in \cite{TarasovVarchenko2000,EtingofVarchenko2002} in the
framework of the theory of intertwiners. As in the case of the
dynamical Weyl group, their restriction to weight spaces are the image
of elements of $U_q\mathfrak{sl}_2$ by embeddings
$U_q\mathfrak{sl}_2\hookrightarrow U_q\mathfrak{sl}_M$ corresponding
to simple roots,\footnote{In \cite{EtingofVarchenko2002} the dynamical
  Weyl group is defined for any Kac--Moody Lie algebra.} {}see \cite[Section 4.1]{EtingofVarchenko2002}. They obey
relations (i) and (iii) of Theorem \ref{t-2} but not (ii). In
other words they define a representation of the braid group $B_M$ on
rational $U$-valued functions. In this representation the squares
$s_i^2$ of the generators act by multiplication by  a rational function on each weight
space.

Let us compare our operators $A_{i,U}(z)$ with the operators
$\mathscr A_{s_i,U}(\lambda)$ of \cite[Section
4.2]{EtingofVarchenko2002} defining the dynamical action of generators,
in the case of $U_q\mathfrak{sl}_M$. The
latter operators are initially defined as functions of sufficiently
large antidominant weights $\lambda$. They extend (upon choosing a
logarithm of $q$) to meromorphic functions of
$\lambda\in \mathfrak h^*$ in the dual of the Cartan subalgebra
$\mathfrak h=\{x\in\mathbb C^M\,|\,\sum_i x_i=0\}$ of
$\mathfrak{sl}_M$.  Let $h_i=(0,\dots,0,1,-1,0,\dots,0)$,
with $1$ in the position $i=1,\dots,M-1$, be the basis of coroots. To relate $z$ to $\lambda$
we identify $\mathfrak h^*$ with $\mathbb C^M/\mathbb C(1,\dots, 1)$
and write $z=q^{2\lambda}=(q^{2\lambda_1},\dots,q^{2\lambda_M})$. This
is defined up to simultaneous scaling of the variables
$z_i$. Since $A_{i,U}(z)$ is a function of the ratios of $z_j$
$A_{i,U}(q^{2\lambda})$ is well-defined.
\begin{prop}\label{p-3}
  \[
    A_{i,U}(q^{2\lambda})|_{U[\mu]}=
    \begin{cases}
      q^{-\mu(h_i)}\mathscr{A}_{s_i,U}(\lambda), &\text{if $\mu(h_i)\geq0$,}
      \\
\displaystyle   (-1)^{\mu(h_i)}   q^{-\mu(h_i)}\frac{[\lambda(h_i)-\mu(h_i)/2]_q}{[\lambda(h_i)+\mu(h_i)/2]_q}
      \mathscr{A}_{s_i,U}(\lambda),&\text{if
        $\mu(h_i)\leq 0$}.
    \end{cases}
  \]
\end{prop}
\begin{proof}
  Since $A_{i,\mu}$ are images of elements of $U_q\mathfrak{sl}_2$ by
  embeddings into $U_q\mathfrak{sl}_M$, it is sufficient to consider
  the case $M=2$ and we set $\lambda=\lambda_1-\lambda_2$ and $z=q^{2\lambda}=z_1/z_2$.
  By complete reducibility we may also assume that
  $U=L_\ell$ is the irreducible representation of $U_q\mathfrak{sl}_2$
  of highest weight $\ell$ and dimension $\ell+1$. It has
  one-dimensional weight spaces $L_\ell[\ell-2k]$ with basis
  $v_{\ell-2k}=F^{(\ell)}v_\ell$, ($k=0,\dots,\ell$), in terms of the
  highest weight vector $v_\ell$.  In this case we have a unique
  generator $s_1$ sending $x\in\mathfrak h^*$ to $-x$.  The following
  formula for $\mathscr A_{s_1,L_\ell}(\lambda)$ can be extracted from
  \cite{EtingofVarchenko2002}, see Corollary 8 (iii), Proposition 12
  and the definition of the dynamical action of the braid group at the
  end of Section 4.2:
  \[
    \mathscr A_{s_1,L_{\ell}}(\lambda)v_{m}= (-1)^kq^{m}\prod_{j=1}^{k}
    \frac{[\lambda+k-j-\ell/2]_q} {[\lambda-k+j+\ell/2]_q}v_{-m},\quad
    m=\ell-2k.
  \]
  By setting $z=q^{2\lambda}$, this can be rewritten as
  \[
    \mathscr A_{s_1,L_{\ell}}(\lambda)v_{\ell-2k}=
    (-1)^kq^{m+k(\ell-k+1)}\prod_{j=0}^{k-1}\frac{1-z
      q^{-\ell+2j}}{1-zq^{\ell-2j}} v_{2k-\ell}.
  \]
  On the other hand, $A_{1,L_\lambda}(z)$ acts on the weight $m$
  subspace of any representation as $A_m(z)$, see \eqref{e-5}.

  It follows from Lemma \ref{l-A1} that
  $A_m(z)=\varphi_m(z)\mathscr A_{s,L_\ell}(q^{2\lambda})$ with
  $\varphi_m(z)=q^{-m}$ for $m\geq0$. If $m\leq0$, \dontprint{
    \begin{align*}
      \varphi_m(z)
      &=
        q^{-m+\frac14(\ell-|m|)(\ell+|m|+2)-\frac14(\ell+|m|)(\ell-|m|+2)}
        \frac{
        \prod_{j=0}^{\frac{\ell-|m|}2-1}\frac{1-zq^{-\ell+2j}}{1-zq^{\ell-2j}}
        }
        {
        \prod_{j=0}^{\frac{\ell+|m|}2-1}\frac{1-zq^{-\ell+2j}}{1-zq^{\ell-2j}}
        }
      \\
      &=
        q^{-m+\frac14(\ell-|m|)(\ell+|m|+2)-\frac14(\ell+|m|)(\ell-|m|+2)}
        \prod_{j=\frac{\ell-|m|}2}^{\frac{\ell+|m|}2-1}\frac{1-zq^{\ell-2j}}{1-zq^{-\ell+2j}}
      \\
      &=\frac{1-zq^{-m}}{1-zq^{m}}.
    \end{align*}
  }
  \[
    \varphi_m(z) =(-1)^m
    \prod_{j=\frac{\ell-|m|}2}^{\frac{\ell+|m|}2-1}\frac{1-zq^{\ell-2j}}{1-zq^{-\ell+2j}}
    =(-1)^m\frac{1-zq^{-m}}{1-zq^{m}}, \quad m\leq 0.
  \]
  so that
  $\varphi_m(q^{2\lambda})=(-1)^mq^{-m}\frac{[\lambda-m/2]_q}
  {[\lambda+m/2]_q}$.
\end{proof}

\subsection{Dynamical Weyl group action for symmetrizable Kac--Moody algebras}{}\label{ss-KacMoody}
Let $U_q\mathfrak g$ be the Drinfeld--Jimbo quantum universal
enveloping algebra of a symmetrizable Kac--Moody algebra $\mathfrak g$
with simple roots $\alpha_1,\dots,\alpha_r$ and simple coroots
$h_1,\dots,h_r$, Cartan matrix $a_{ij}=\alpha_j(h_i)$ and
$d_i\in\mathbb Z_{\geq1}$ so that $(d_ia_{ij})$ is symmetric. Let
$Q=\oplus_{i=1}^r\mathbb Z\alpha_i$ be its root lattice and $T$ the
torus $\operatorname{Hom}(Q,\mathbb C^\times)$ of characters of $Q$.
The Weyl group
$\mathbb W$ acts on $R$ and $T$.  Then for each simple root
$\alpha_i$, we have an embedding
$j_i\colon U_{q_i}\mathfrak{sl}_2\to U_q\mathfrak g$ with $q_i=q^{d_i}$
and an evaluation map
$z_{\alpha_i}\in\operatorname{Hom}(T,\mathbb C^\times)$ so that the
ring of regular functions on $T$ is
$\mathbb C[z_{\alpha_1}^{\pm1},\dots,z_{\alpha_r}^{\pm1}]$.

The following generalization of Corollary \ref{c-1}, which was
formulated as a conjecture in a previous version of this work, is a
refinement of the construction of Etingof and Varchenko of a dynamical
action of the {\em braid} group on rational functions with values in an
integrable representation.  A proof by Anfisa Gurenkova is included in
Appendix \ref{a-2}.

\begin{theorem}\label{t-KacMoody}
  Let $U$ be an integrable representation of $U_q\mathfrak g$ and
  for each $i=1,\dots, r$ let $A_{i,U}(z)$ be the $\mathrm{End}(U)$-valued
  rational function on $T$ such that on the weight space of weight $\mu$
  \[
    A_{i,U}(z)|_{U[\mu]}=j_i(A_{\mu(h_i)}(z_{\alpha_i}))\colon U[\mu]\to U[s_j\mu],
  \]
  where $A_m(z)$ is defined in \eqref{e-5}.
  Then the action of simple reflections
  \[
    (s_if)(z)=A_{i,U}(s_iz)f(s_iz),\quad i=1,\dots r,
  \]
  defines a representation of $\mathbb W$ on rational functions on $T$
  with values in $U$.    
\end{theorem}

\subsection{\texorpdfstring{$B$}{B}-operators, quantum Weyl group, Yangian limit}
The limiting values of $R$-matrices as the spectral parameter tends to $0$ or $\infty$ are
solutions of the Yang--Baxter equations without spectral parameter. They are $R$-matrices
for quantum enveloping algebras of finite dimensional Lie algebras.
By dividing the $A$-operators by their limiting values we obtain weight zero ``$B$-operators''
\cite{EtingofVarchenko2002}.
Define a formal series $B_m(z)$ by the formula
\begin{equation}\label{e-7}
  A_m(z)=A_{m}(0)B_m(z).
\end{equation}
The series $B_m(z)$ is well-defined as an endomorphism of the
$m$-weight space of any finite dimensional representation and one has
a universal formula:
\begin{prop}\label{p-6}\
  \begin{enumerate}
    \item[(i)]
  For $m\geq0$,
  \[
    B_m(z)=\sum_{j=0}^\infty
    q^{\frac{j(j-3)}2}\frac{(-z)^j}{[j]_q!}F^jE^j\prod_{i=1}^j\frac{1-q^2}{1-z
      q^{m+2i}}
  \]
    \item[(ii)]
  For $m\leq0$,
  \[
    B_m(z)=\sum_{j=0}^\infty
    q^{\frac{j(j-3)}2}\frac{(-z)^j}{[j]_q!}E^jF^j\prod_{i=1}^j\frac{1-q^2}{1-z
      q^{-m+2i}}
  \]
\end{enumerate}
\end{prop}
\begin{proof}
  The first formula can be checked directly or deduced from the remark
  after Corollary 40 in \cite{EtingofVarchenko2002}, by using
  Proposition \ref{p-3}. There is a similar formula in {\it
    loc.~cit.}, Proposition 14, but probably the powers of $q$ need to
  be adjusted there. The second formula follows by applying the Cartan involution, see
  Corollary \ref{c-A1}
\end{proof}
For $m=k-k'$, these operators define via Howe duality
$U_q\gl_N$-linear endomorphisms of
$\bigwedge^{k} V(z_1)\otimes\bigwedge^{k'}V(z_2)$. Their
representation theoretical meaning in terms of $U_qL\gl_N$ is not
clear but in the limit $q\to 1$, $A_{m}(0)$ tends up to signs to the
permutation $P$ and in a suitable $z\to 1$ limit, $B_m(z)$ converges
to a rational $R$ matrix, intertwining the coproduct and the opposite
coproduct of the Yangian. More precisely we let
$q=\exp(\epsilon\hbar)$, $z=\exp(2\epsilon u)$ and consider the limit
$\epsilon\to0$:
\begin{theorem}\label{t-3}
  Let $\check R_{k,k'}(z,q)$ be the family of braiding matrices of
  Theorem \ref{t-1}, viewed as a linear map on the underlying vector
  space $\bigwedge^kV\otimes \bigwedge^{k'}V$.  Let
  \[
    R_{k,k'}(z,q)=P\check R_{k,k'}(z,q)\in\textstyle{\operatorname{End}_{\mathbb
    C}(\bigwedge^kV\otimes \bigwedge^{k'}V)},\text{ with
    $Pv'\otimes v=(-1)^{kk'}v\otimes v'$,}
\]
be the corresponding
  $R$-matrix, normalized to be the identity on the product of highest
  weight spaces.  Then the limit
  \[ R^{\mathrm{rat}}_{k,k'}(u,\hbar)=\lim_{\epsilon\to
0}R_{k,k'}(e^{\hbar \epsilon},e^{2\epsilon u})
  \] exists and is a $\gl_N$-invariant solution of the Yang--Baxter
equation. Moreover
  \[ R^{\mathrm{rat}}_{k,k'}(u,\hbar)=
\sum_{j\geq0}\frac{(-\hbar)^j}{j!}\prod_{i=1}^j\frac1{u+\hbar(i+|k-k'|/2)}
    \begin{cases} F^jE^j&\text{if $k\geq k'$,} \\ E^jF^j&\text{if
$k\leq k'$.}
    \end{cases}
  \] Here $E=\sum_{i=1}^N\psi^*_i\otimes\psi_i$,
$F=-\sum_{i=1}^N\psi_i\otimes\psi_i^*$ are the $q\to 1$ limits of the
$U_q\mathfrak{sl}_2$-generators of Section \ref{s-2}.
\end{theorem}
\begin{proof} The limit exists since both $E, F$ and the coefficients
in the formula for $\check R_{k,k'}$ have a limit as $\epsilon\to
0$. The Yang--Baxter equation then follows from the braiding relation \eqref{e-3}
of $\check R_{k,k'}$.  The action of the subalgebra $U_q\gl_N$ does
not involve the spectral parameter and becomes an action of $U\gl_N$,
so that the limit is $\gl_N$-invariant.

{}By Theorem \ref{t-3} $\check R_{k,k'}(z,q)$ is a rational function of $z$
with poles at finitely many integers powers of $q$. So if we take $q$ to be close
to 1,  $\check R_{k,k'}(z,q)$ is holomorphic in, say,
the disk $|z|<1/2$. Its value at $q=1$ is independent of $z$ and is a
homomorphism of $L\gl_N$-modules, which are irreducible for
$z=z_1/ z_2\neq1$. By Schur's lemma
$\check R_{k,k'}(z,1)=\check R_{k,k'}(0,1)$ must be proportional to
the homomorphism $P\colon u\otimes v\mapsto (-1)^{kk'}v\otimes
u$. With the normalization of Theorem \ref{t-1}, the factor of proportionality is 1.
It  follows that
\[
  R^{\mathrm{rat}}_{k,k'}(u,\hbar) = \lim_{\epsilon\to 0}\check
  R_{k,k'}(0,e^{\epsilon \hbar})^{-1} \check R_{k,k'}(e^{2\epsilon
    u},e^{\epsilon\hbar}).
\]
By Theorem \ref{t-1} this can be expressed in terms of
$A_m(z)=A_m(z,q)$ for $m=k-k'$:
\[
  R^{\mathrm{rat}}_{k,k'}(u,\hbar) = \lim_{\epsilon\to
    0}A_m(0,e^{\epsilon\hbar})^{-1}A_m(e^{2\epsilon
    u},e^{\epsilon\hbar}) = \lim_{\epsilon\to 0}B_m(e^{2\epsilon
    u},e^{\epsilon\hbar}).
\]
Then, in the limit $\epsilon\to 0$, $E=\sum\psi_i\otimes\psi_i^*$ and
$F=-\sum\psi_i^*\otimes\psi_i$ obey $\mathfrak{sl}_2$ commutation
relations and, by Proposition \ref{p-6}, $B_m$ tends to
\[
  \sum_{j=0}^\infty\frac{(-\hbar)^j}{j!}F^jE^j\prod_{i=1}^j\frac1{u+\hbar(i+m/2)}
\]
for $m\geq0$ and
\[
  \sum_{j=0}^\infty\frac{(-\hbar)^j}{j!}E^jF^j\prod_{i=1}^j\frac1{u+\hbar(i-m/2)}
\]
for $m\leq 0$.
\end{proof}
\begin{remark} If $k=k'$ we recover the formula proposed by Smirnov, see \cite{Smirnov2016}, eq.~(115).
  Understanding this formula was part of the motivation of the present work. For general $k,k'$ our formulas seem
  to be related to the formula in \cite{Smirnov2016} by a shift of spectral parameter which could be partly explained by
  a redefinition of exterior powers.
\end{remark}
\begin{remark}
  Another way to arrive at the formula for $R^{\mathrm{rat}}$ of Theorem \ref{t-3}
  is to view it as the $R$-matrix for Yangians and reproduce  the proof of Theorem
  \ref{t-1} in the rational case, by using the Yangian version of the Date--Okado formula
  \eqref{e-DateOkado}. A direct proof of the latter is found in \cite{TarasovUvarov2020-2},
  see formula (4.14).
\end{remark}

\appendix
\section{\texorpdfstring{$U_q\mathfrak{sl}_2$}{Uqsl2} yoga}\label{a-1}
We recall some basic facts about $U_q\mathfrak{sl}_2$,
see e.g. \cite{ChariPressley1994}, and
apply them to matrix elements of $A_m(z)$ and $B_m(z)$.
The irreducible representations of $U_q\mathfrak{sl}_2$  for
$q$ not a root of unity are deformations of those of $\mathfrak{sl}_2$.
The irreducible representation $L_\ell$ of highest weight $\ell\in\mathbb Z_{\geq0}$
has a basis $v^\ell_\ell,v^\ell_{\ell-2},\cdots,v^\ell_{-\ell}$ of weight vectors
with $v^\ell_{\ell-2k}=F^{(k)}v^\ell_\ell$.

The Cartan involution is the algebra automorphism of $U_q\mathfrak{sl}_2$
such that $\theta(E)=-F,\theta(F)=-E,\theta(K)=K^{-1}$. It is a coalgebra
antihomomorphism and for $v\in L_\ell$, $X\in U_q\mathfrak{sl}_2$,
$\theta(X)v=s_\ell X s_\ell^{-1}v$, where
\[
  s_\ell v^\ell_m=(-1)^{(\ell-m)/2}v^\ell_{-m}.
\]

Let
  $C\in Z(U_q\mathfrak{sl}_2)$ be the Casimir element
  \[
    C=EF+\frac{q^{-1}K+qK^{-1}}{(q-q^{-1})^2}=FE+\frac{qK+q^{-1}K^{-1}}{(q-q^{-1})^2},
  \]
  belonging to the centre of $U_q\mathfrak{sl}_2$. Its value in
  $L_\ell$ is
  \[
    c_\ell=\frac{q^{\ell+1}+q^{-\ell-1}}{(q-q^{-1})^2}.
  \]
  One has the well-known identity, easily checked by induction,
  \begin{equation}\label{e-id}
    E^jF^j=\prod_{i=0}^{j-1} \left(
      C-\frac{q^{2i+1}K^{-1}+q^{-2i-1}K}{(q-q^{-1})^2} \right).
  \end{equation}
  \begin{lemma}\label{l-A1} Let $A_m(z)$ be the formal power series
    \eqref{e-5} and $B_m(z)$ the formal power series defined by \eqref{e-7}. Then
  \begin{align*}
    A_m(z)v_m^\ell&=\prod_{j=0}^{\frac{\ell-|m|}2-1}
    \frac{z- q^{\ell-2j}}{1-z q^{\ell-2j}}v_{-m}^\ell
  \\
    B_m(z)v_m^\ell&=\prod_{j=0}^{\frac{\ell-|m|}2-1}
    \frac{1-z q^{-\ell+2j}}{1-z q^{\ell-2j}}v_{m}^\ell
  \end{align*}
\end{lemma}
\begin{proof}
  Let us first consider the case $m\geq0$.

  The identity \eqref{e-id} implies that
  on the subspace of $L_\ell$ of weight $-\ell+2k$, we have
  \begin{align*}
    E^jF^j|_{L_{\ell}[-\ell+2k]}
    &=\prod_{i=0}^{j-1}
      \frac{q^{\ell+1}+q^{-\ell-1}-q^{2i+1+\ell-2k}-q^{-2i-1-\ell+2k}}{(q-q^{-1})^2}
    \\
    &=\prod_{i=0}^{j-1}
      \frac{(q^{\ell+1-k+i}-q^{-\ell-1+k-i})(q^{k-i}-q^{-k+i})}{(q-q^{-1})^2}
    \\
    &=\prod_{i=0}^{j-1}
      [\ell+1-k+i]_q[k-i]_q.
  \end{align*}
  It follows that for $m=\ell-2k\geq0$
  \[
    A_m(z)=\sum_{j=0}^{k}(-q)^j\frac{1-q^mz}{1-q^{2j+m}z}\frac1{[j]_q![j+m]_q!}
    \prod_{i=0}^{j-1}[\ell+1-k-i]_q[k-i]_qF^m
  \]
  Therefore,
  \begin{align*}
    A_m(z)v^\ell_{m}
    &= \sum_{j=0}^{k}(-q)^j\frac{1-q^mz}{1-q^{2j+m}z}\frac{[\ell-k]_q!}{[j]_q![j+m]_q![k]_q!}
      \prod_{i=0}^{j-1}[\ell+1-k-i]_q[k-i]_qv^\ell_{-m}
    \\
    &= \sum_{j=0}^{k}(-q)^j\frac{1-q^{\ell-2k}z}{1-q^{2j+\ell-2k}z}\frac{\prod_{i=0}^{k-1}[\ell-k+j-i]_q}{[j]_q![k-j]_q!}v^\ell_{-m}.
  \end{align*}
  We thus have a rational function of $z$ with simple poles at
  $z=q^{-\ell+2k-2},$ $q^{-\ell+2k-4},$ $\dots,q^{-\ell}$ which is bounded
  at infinity. By comparing the residues at the poles and the value
  $\genfrac[]{0pt}{1}{\ell-k}{k}_q$ at $z=q^{-\ell+2k}$, we see that
  \[
    A_m(z)v^\ell_m=
    \prod_{j=0}^{k-1}\frac{z-q^{\ell-2j}}
    {1-zq^{\ell-2j}}v^\ell_{-m},\quad \text{if $m=\ell-2k\geq0$.}
  \]
  For $m\leq 0$, the inversion relation implies that
  $A_{m}(z)v^\ell_m=A_{|m|}(z^{-1})^{-1}v^\ell_{m}$. With $-m=\ell-2k\geq0$ we
  find that
  \begin{align*}
    A_m(z)v^\ell_m
    &=\prod_{j=0}^{k-1}
      \frac{1-z^{-1}q^{\ell-2j}}{z^{-1}-q^{\ell-2j}}v^\ell_{-m}
    \\
    &=
      \prod_{j=0}^{k-1}\frac{z-q^{\ell-2j}}{1-zq^{\ell-2j}}v^\ell_{-m},
      \quad \text{if $-m=\ell-2k\geq0$.}
  \end{align*}
  This proves the formula for $A_m(z)$ since $k=\frac{\ell-|m|}2$. In
  particular, $A_m(0)v^\ell_m=\prod_{j=0}^{k-1}(-q^{\ell-2j})v^\ell_{-m}$, and
  the formula for $B_m(z)=A_m(0)^{-1}A_m(z)$ follows.
\end{proof}
\begin{cor}\label{c-A1}
  Let $U$ be any finite dimensional representation of $U_q\mathfrak{sl}_2$. Then for any $m\in\mathbb Z$
  \[
    \theta(B_m(z))|_{U[-m]}=B_{-m}(z)|_{U[-m]}
  \]
\end{cor}
\begin{proof}
  It is sufficient to prove this for $U$ irreducible. In this case the statement follows from Lemma \ref{l-A1} and
  the identity $\theta(X)=s_\ell X s_\ell^{-1}$ for the action of $X\in U_q\mathfrak{sl}_2$ on $L_\ell$.
\end{proof}
\section{Proof of Theorem \ref{t-KacMoody}, by Anfisa Gurenkova}\label{a-2}

We use the notations of Section \ref{ss-KacMoody} and adopt the conventions of \cite{EtingofVarchenko2002}. We denote by $\mathbb W$ and
$\Tilde{\mathbb W}$ the Weyl group and the braid group of the root system of the symmetrizable Kac--Moody algebra $\mathfrak g$, respectively.
For each element $w\in\mathbb{W}$ and an integrable $U_q\mathfrak{g}$-module $V$ Etingof and Varchenko \cite{EtingofVarchenko2002} define an operator $\mathscr{A}_{w, V}(\lambda, \mu): V[\mu]\rightarrow V[w\mu]$. 
The following is true:

\begin{prop}[Lemma 17 in \cite{EtingofVarchenko2002}]
  \label{EV braid relation of dynWeyl}
If $l(w_1 w_2)=l(w_1)+l(w_2)$, then
\begin{equation}\label{action of A_w1w2 via A_w1 and A_w2}
    \mathscr{A}_{w_1w_2, V}(\lambda, \mu) = \mathscr{A}_{w_1, V}(w_2\lambda, w_2\mu) \mathscr{A}_{w_2, V}(\lambda, \mu) 
\end{equation}
\end{prop}




The definition of the operators $\mathscr{A}_w$ is consistent with restricting to the root $\fsl_2$ in the following sense (see \cite[Section 4.1]{EtingofVarchenko2002}).

\begin{prop}
    For any $i$:
    \begin{equation}\label{A_si via sl_2}
        \mathscr{A}_{s_i, V}(\la, \mu) = \mathscr{A}_{s, V'}(\lambda(h_i),\mu(h_i)),
    \end{equation}
    where $V'$ is the $U_{q_i}(\mathfrak{sl}_2)$-module pulled back by the embedding $j_i$.
\end{prop}

Due to these two facts, it is enough to construct the operators {}$\mathcal A_{s,V}(l,m)$ for $U_q\fsl_2$-modules $V$. Then one defines an operator $\mathscr{A}_{w,V}$ by decomposing $w$ into simple reflections \eqref{action of A_w1w2 via A_w1 and A_w2} and then using \eqref{A_si via sl_2}. Also one can define $\mathscr{A}_{w,V}$ for any $w\in\Tilde{\IW}$ by this procedure.

By Proposition {}\ref{EV braid relation of dynWeyl},
the operators $\mathscr{A}_{w, V}$ define an action of $\Tilde{\IW}$ on $V$-valued functions of $\la$ as follows:

\begin{equation}\label{action of A on functions}
    (w * u)(\la) = \mathscr{A}_{w,V}(w^{-1}(\la)) u(w^{-1}\la).
\end{equation}

Let $P\Tilde{\mathbb{W}}$ be the pure braid group, i.e. the kernel of the map $\Tilde{\mathbb{W}}\rightarrow \mathbb{W}$. It is a normal subgroup in $\Tilde{\mathbb{W}}$ generated by $s_i^2$. Notice that  $\mathscr{A}_{s_i}(s_i\la,s_i\mu)\mathscr{A}_{s_i}(\la,\mu)$ is a weight zero operator. Moreover, it is diagonal in a basis of $V[\mu]$ which is consistent with the decomposition of $V$ into irreducible $j_i(U_{q_i}(sl_2))$-modules.
We will see in Lemma \ref{character of pure braid} that in fact it acts by a scalar (depending on $\la$, $\mu$, and $q$). Therefore it defines a character $\chi_V: P\Tilde{\mathbb{W}} \rightarrow Fun(\la,\mu,q)$ with values in the meromorphic functions in $\la$, $\mu$, and $q$.

Our goal is to modify {}the action \eqref{action of A on functions} by multiplying each $\mathscr{A}_w(\la,\mu)$ by a suitable function of $\la, \mu, q$ so that it factors through $\IW$.

First, we find the type of functions which do not break the braid relation. Second, we compute the character and see that it belongs to this type of functions.

\begin{lemma}\label{renormalization does not affect braid relations}

Let \[\Tilde{\mathscr{A}_{s_i}}(\la,\mu) = f(\la(h_i), \mu(h_i),q_i )\mathscr{A}_{s_i}(\la,\mu)\] for some function $f(l,m,q)$. Then the operators $\Tilde{\mathscr{A}_{s_i}}(\la,\mu)$ satisfy the braid relations.
\end{lemma}

\begin{proof}
For any two nodes $1,2$ of the Dynkin diagram consider the corresponding simple reflections $s_1, s_2$. Either there is no relation between them, or they satisfy the braid relation $s_i\cdots s_2 s_1 s_2  = s_j\cdots s_1 s_2 s_1$. Here the number of factors on both sides is equal to some integer $m_{12}$, depending on the number of edges between the nodes $1$ and $2$,
and  $(i,j)=(1,2)$ or $(2,1)$ depending on the parity of $m_{12}$. Denote this product by $w$. Since $l(w) = m_{12}$, by Proposition \ref{EV braid relation of dynWeyl} we have 

\begin{multline}\label{braid equation}
    \mathscr{A}_{s_i, V}(s_j\cdots s_2(\lambda, \mu))\cdots\mathscr{A}_{s_1, V}(s_2(\lambda, \mu)) \mathscr{A}_{s_2, V}(\lambda, \mu)  = \\ = \mathscr{A}_{s_j, V}(s_i\cdots s_1(\lambda, \mu))\cdots\mathscr{A}_{s_2, V}(s_1(\lambda, \mu)) \mathscr{A}_{s_1, V}(\lambda, \mu) 
\end{multline}

When we replace $\mathscr{A}$-s by $\Tilde{\mathscr{A}}$-s in the  LHS of \eqref{braid equation}, it is multiplied by 

\[f\Big(  s_j\cdots s_2(\lambda, \mu)(h_i),q_i\Big)\cdots f\Big(  s_2(\lambda, \mu)(h_1),q_1\Big) f\Big ((\lambda, \mu)(h_2),q_2\Big), \]

The following fact applied to a subalgebra corresponding to the pair of roots $\alpha_1,\alpha_2$ finishes the proof:

\begin{lemma}
    Let $\mathfrak{g}$ be a finite-dimensional Lie algebra of rank 2 and $w_0=s_{i_1}\cdots s_{i_l}$ be a reduced word for the longest element in $\mathbb{W}$. Then for any weight $\la$ the set \{$\big(s_{i_{k+1}}\cdots s_{i_l}\la(h_{i_k}),d_{i_k}\big)| k=1\ldots l-1$\} does not depend on a reduced word.
\end{lemma}
\begin{proof}
    It is a reformulation of Lemma 2 in \cite{EtingofVarchenko2002}. For reader's convenience we reproduce its proof here.

    Let $(\ ,\ )$ be the non-degenerate bilinear form on $\fh$ such that $(h, h_i) = d^{-1}_i\alpha_i(h)$.        
    Notice that $s_{i_{k+1}}\cdots s_{i_l}\la(h_{i_k}) = \la(s_{i_l}\cdots s_{i_{k+1}} h_{i_k})$.
    Denote $h^k := s_{i_l}\cdots s_{i_{k+1}} h_{i_k}$. Also notice that $d_i = \frac{2}{(h_i, h_i)} = \frac{2}{(wh_i,wh_i)}$ for any $w\in \IW$.
    Since $w_0$ is the longest element, $\{h^k\}$ is the set of all the positive coroots, each appearing exactly once. 
    Therefore the set in the statement is $\{\big(\la(h^l), \frac{2}{(h^l,h^l)}\}$ which does not depend on a reduced decomposition. 
\end{proof}

\end{proof}

\begin{lemma}\label{character of pure braid}
The character $\chi_{V}: P\Tilde{\mathbb{W}} \rightarrow Fun(\la,\mu,q)$ is defined by

    \[\chi_{V}(s_i^2) = (-1)^{\mu(h_i)}\frac{[(\la+\frac{\mu}{2})(h_i)]_{q_i}}{[(\la-\frac{\mu}{2})(h_i)]_{q_i}}\]
\end{lemma}

\begin{proof}
    By the consistency with the restriction (\ref{A_si via sl_2}) it is enough to prove the statement for $U_q\fsl_2$. Let $V_n$ denote the irreducible $\fsl_2$-module of the highest weight $n$. By corollary 8, Proposition 12 and the definitions of section 4.2 of \cite{EtingofVarchenko2002},  \[\mathscr{A}_{s, V_n[n-2k]}(l, n-2k) = A^\infty_{V_n[n-2k]} B_{V_n[n-2k]},\] where $A^\infty_{V_n[n-2k]}v_{n-2k} = (-1)^kq^{n-2k}v_{2k-n}$ and $B_{V_n}$ is a weight zero operator defined by

\[B_{V_n[n-2k]}(l) = \prod_{j=1}^k\frac{[l+1+j]_q}{[l+1-n+2k-j]_q}.\]

Therefore for $m=n-2k$
 
    \begin{align*}
      \chi_{V_n}(s^2)(l,m)
      &= \mathscr{A}_{s, V_n}(-l,-m) \mathscr{A}_{s, V_n}(l, m)
      \\
      &= B_{V_n[-m]}(l-1+\frac{-m}{2}) B_{V_n[m]}(-l-1+\frac{m}{2})\cdot(-1)^n
      \\
      &=(-1)^m\prod_1^{n-k}\frac{[l+\frac{-n+2k}{2}+j]_q}{[l-\frac{-n+2k}{2}-j]_q}\prod_1^k\frac{[-l+\frac{n-2k}{2}+j]_q}{[-l-\frac{n-2k}{2}-j]_q}
      \\
      &=
        (-1)^m\frac{\prod_{k+1}^{n}[l-\frac{n}{2}+j]_q\prod_0^{k-1}[l-\frac{n}{2}+j]_q}{\prod_{k+1}^{n}[l+\frac{n}{2}-j]_q\prod_0^{k-1}[l+\frac{n}{2}-j]_q}
      \\
      &=(-1)^m\prod_{-\frac{n}{2}}^{\frac{n}{2}}\frac{[l+j]_q}{[l-j]_q}\cdot\frac{[l+\frac{n}{2}+k]_q}{[l-\frac{n}{2}-k]_q} = (-1)^n\frac{[l+\frac{m}{2}]_q}{[l-\frac{m}{2}]_q}
    \end{align*} 
\end{proof}

\begin{prop}
    Let \[A_{s_i, V[\mu]} = \begin{cases}
    q_i^{-\mu(h_i)}\mathscr{A}_{s_i,V[\mu]}, & \text{if } \mu(h_i)\geq 0, \\
    (-1)^{\mu(h_i)}q_i^{-\mu(h_i)}\frac{[(\la-\frac{\mu}{2})(h_i)]_{q_i}}{[(\la+\frac{\mu}{2})(h_i)]_{q_i}}\mathscr{A}_{s_i, V[\mu]}, & \text{if } \mu(h_i)< 0\end{cases}\]

    Then $A_{s_i}$ define the action of the Weyl group on $V-$valued functions of $\la$ via formulas \eqref{action of A_w1w2 via A_w1 and A_w2} and \eqref{action of A on functions}.
\end{prop}

\begin{proof}
    From Lemma \ref{renormalization does not affect braid relations} it follows that $A_{s_i}$ satisfy the braid relations. It suffices to check the relation $(s_i*)^2 = 1$ for $U_q\fsl_2$. In this case

    \[A_{s, V_n}(-l,-m)A_{s,V_n}(l,m) = (-1)^n\frac{[-l+\frac{m}{2}]_q}{[-l-\frac{m}{2}]_q} \cdot \chi_{V_n}(s^2)(l,m) = 1\]
\end{proof}

\begin{remark}
    In the case $\fg=\fsl_n$ the operators $A_{s_i, V[\mu]}$ coincide with the operators defined in Proposition \ref{p-3}.
  \end{remark}

  \section*{Conflict of interest}
  The authors have no conflict of interest to declare that are relevant to this article.
\newpage

\printbibliography

\end{document}